\documentclass[a4paper,reqno]{amsart}
\usepackage{amssymb}
\usepackage[utf8]{inputenc}
\usepackage[T1]{fontenc}
\usepackage{ae,icomma,graphicx}
\usepackage{enumerate,titletoc,setspace}
\usepackage[b]{esvect}
\usepackage[pagebackref,colorlinks,linkcolor=red,citecolor=blue,urlcolor=blue,hypertexnames=true]{hyperref}

\newtheorem{theorem}{Theorem}[subsection]
\newtheorem{lemma}[theorem]{Lemma}

\newtheorem{proposition}[theorem]{Proposition}
\newtheorem{prop}[theorem]{Proposition}
\newtheorem{cor}[theorem]{Corollary}

\newtheorem{corollary}[theorem]{Corollary}

\theoremstyle{definition}
\newtheorem{definition}[theorem]{Definition}

\newtheorem{remark}[theorem]{Remark}
\newtheorem{rem}[theorem]{Remark}
\newtheorem{example}[theorem]{Example}

\newcommand\N{\mathbb N}

\newcommand\R{\mathbb R}

\newcommand\ph\varphi
\newcommand\ps\psi
\newcommand\ep\varepsilon
\newcommand\rh\varrho
\newcommand\al\alpha
\newcommand\be\beta
\newcommand\ga\gamma
\newcommand\om\omega
\newcommand\ta\tau
\renewcommand\th\vartheta
\newcommand\de\delta
\newcommand\ze\zeta
\newcommand\ch\chi
\newcommand\et\eta
\newcommand\io\iota
\newcommand\la\lambda
\newcommand\si\sigma

\newcommand\Ga\Gamma
\newcommand\De\Delta
\newcommand\Th\Theta
\newcommand\La\Lambda
\newcommand\Si\Sigma
\newcommand\Ph\Phi
\newcommand\Ps\Psi
\newcommand\Om\Omega

\newcommand{\x}{{\tt x}}
\newcommand\ac{[\x]}
\newcommand\cx{[\x]}
\newcommand\rx{\R\ac}
\newcommand\RR{\R}
\newcommand\rxlin{\rx_1}
\newcommand\sos{\sum\rx^2}
\newcommand\nn{_{\ge0}}
\newcommand\pos{_{>0}}

\newcommand\sa{S\R^{\al\times\al}}

\newcommand\sapd{S\R^{\al\times\al}_{\succ0}}

\newcommand\sapsd{S\R^{\al\times\al}_{\succeq0}}
\renewcommand\emptyset\varnothing
\newcommand\rad[1]{\sqrt{#1}}
\newcommand\rrad[1]{\sqrt[R]{#1}}
\newcommand\s[1]{s(#1)}

\def\x{{\tt x}}

\def\rx{\R[\x]}
\def\cI{\mathcal I}
\newcommand\vecx[1]{{\vv{[\text{{\tt\normalsize \x}}]}_#1}}

\DeclareMathOperator\Span{span}

\DeclareMathOperator\im{im}
\DeclareMathOperator\ev{ev}
\DeclareMathOperator\dist{dist}
\DeclareMathOperator\supp{supp}
\DeclareMathOperator\tr{tr}

\renewcommand{\setminus}{\smallsetminus}

\newcommand{\eop}{\hfill$\Box$}

\def\Rmm{\R^{m \times m} }

\def\bes{\begin{equation*}}
\def\ees{\end{equation*}}

\def\beq{\begin{equation} }
\def\eeq{\end{equation} }

\def\ben{\begin{enumerate} }
\def\een{\end{enumerate} }

\title[A sums of squares dual for SDP and infeasible LMI]{An exact duality theory for semidefinite programming based on sums of squares}

\author[Igor Klep]{Igor Klep}
\address{Igor Klep, Department of Mathematics, 
The University of Auckland, Private Bag 92019, Auckland 1142, New Zealand}
\email{igor.klep@auckland.ac.nz}
\author[Markus Schweighofer]{Markus Schweighofer}
\address{Markus Schweighofer,
Universit\"at Konstanz,
Fachbereich Mathematik und Statistik,
78457 Konstanz, Allemagne}
\email{markus.schweighofer@uni-konstanz.de}
\date{November 10, 2012}
\thanks{This research was supported through the programme ``Research in Pairs'' (RiP) by the Mathematisches Forschungsinstitut Oberwolfach in 2010.
The first author was 
supported by the Faculty Research Development Fund (FRDF) of The
University of Auckland (project no. 3701119), and partially
 supported by the Slovenian Research Agency under Project no. J1-3608 and Program no. P1-0222. 
Part of the research
was done while the first author was on leave from the University of Maribor, and held a visiting professorship at the Universität Konstanz in 2011}

\begin{document}

\setcounter{tocdepth}{2}
\contentsmargin{2.55em} 
\dottedcontents{section}[3.8em]{}{2.3em}{.4pc} 
\dottedcontents{subsection}[6.1em]{}{3.2em}{.4pc}

\makeatletter
\newcommand{\mycontentsbox}{%
{\centerline{NOT FOR PUBLICATION}
\small\tableofcontents}}
\def\enddoc@text{\ifx\@empty\@translators \else\@settranslators\fi
\ifx\@empty\addresses \else\@setaddresses\fi
\newpage\mycontentsbox}
\makeatother

\begin{abstract}
Farkas' lemma is a fundamental result from \emph{linear programming} providing \emph{linear} certificates for infeasibility of systems of linear inequalities.
In \emph{semidefinite programming}, such linear certificates only exist for \emph{strongly} infeasible linear matrix inequalities. 
We provide \emph{nonlinear algebraic} certificates for \emph{all} infeasible linear matrix inequalities in the spirit of real algebraic geometry:
A linear matrix inequality $A(x)\succeq 0$ is infeasible if and only if $-1$ lies in the quadratic module associated to $A$.
We also present a new exact duality theory for 
semidefinite programming, motivated by 
the real radical and
 sums of squares certificates from real algebraic geometry.
\end{abstract}

%%%%%%%%%%%
\iffalse

Farkas' lemma is a fundamental result from linear programming providing linear certificates for infeasibility of systems of linear inequalities. In semidefinite programming, such linear certificates only exist for strongly infeasible linear matrix inequalities. We provide nonlinear algebraic certificates for all infeasible linear matrix inequalities in the spirit of real algebraic geometry:
A linear matrix inequality "A(x) is positive semidefinite" is infeasible if and only if -1 lies in the quadratic module associated to A(x). We also present a new exact duality theory for semidefinite programming, motivated by the real radical and sums of squares certificates from real algebraic geometry.

\fi
%%%%%%%%%%%
\maketitle

\section{Introduction}\label{sec:intro}

A \emph{linear matrix inequality} (LMI) is a condition of the form
\[
A(x) = A_0 + \sum_{i=1}^n x_i A_i \succeq0\qquad(x\in\R^n)
\]
where the $A_i$ are symmetric matrices of the same size and one is interested in the solutions $x\in\R^n$ making $A(x)$ positive semidefinite ($A(x)\succeq0$).
The solution set to such an inequality is a closed convex semialgebraic subset of $\R^n$ called a {\em spectrahedron} or an \emph{LMI domain}. Optimization of linear
objective functions over spectrahedra is called \emph{semidefinite programming} (SDP) \cite{bv96,to01,wsv00}, and is a subfield of convex optimization. 
In this article, we are  concerned with the duality theory of SDP from 
a viewpoint of a real algebraic geometer, 
and with the important
SDP feasibility problem: When is an LMI \emph{feasible}; i.e., when is there
an $x\in\R^n$ satisfying $A(x)\succeq0$?

A diagonal LMI, where all $A_i$ are diagonal matrices, is just a system of linear inequalities, 
and its solution set 
 is a polyhedron. Optimization of linear
objective functions over polyhedra is called \emph{linear programming} (LP).
The ellipsoid method developed by Shor, Yudin, Nemirovskii and Khachiyan
showed at the end of the 1970s for the first time that the LP feasibility problem (and actually the problem of solving LPs) can be solved in
(deterministically) polynomial time (in the bit model of computation assuming rational coefficients) \cite[Chapter 13]{sr86}. Another breakthrough
came in the 1980s with the introduction of the more practical interior point methods by Karmarkar and their theoretical underpinning by
Nesterov and Nemirovskii \cite{nn94,ne07}.

The motivation to replace the prefix ``poly'' by ``spectra'' is to replace ``many'' values of linear polynomials (the diagonal values of $A(x)$) by  the ``spectrum'' of $A(x)$ (i.e., the set of its eigenvalues).
The advantage of LMIs over systems of linear inequalities (or of spectrahedra over polyhedra, and SDP over LP, respectively) is a considerable gain of
expressiveness which makes LMIs an important tool in several areas of applied and pure mathematics.
Many problems in control theory, system identification and signal processing can be formulated using
LMIs \cite{befb94,par00,hg05,du07,ce10}. Combinatorial optimization problems can often be modeled or approximated by SDPs \cite{go97}.
LMIs also find application in real algebraic geometry for finding sums of squares decompositions of polynomials \cite{las10,ma08}.
There is even a hierarchy of SDP approximations to polynomial optimization problems \cite{lau09} consisting of the so-called
Lasserre moment relaxations. 
In this article, rather than trying to solve polynomial optimization problems by
using SDP, we borrow ideas and techniques from real algebraic geometry and polynomial optimization in order to get new results
in the theory of semidefinite programming.

The price to pay for the increased expressivity of SDP is that they enjoy some less good properties. 
First of all, the complexity of solving general SDPs is a very subtle issue.
For applications in combinatorial optimization, it typically follows  from the  theory of the
ellipsoid method \cite{sr86} or interior point methods \cite{nn94} that the translation into SDP yields a polynomial time algorithm (see \cite[Section 1.9]{dk02} for exact statements).
However, the complexity status of the LMI feasibility problem (the problem of deciding
whether a given LMI with rational coefficients has a solution) is largely unknown. What is known is essentially
only that (in the bit model) LMI feasibility lies
either in $\text{NP}\cap\text{co-NP}$ or outside of $\text{NP}\cup\text{co-NP}$. Therefore it cannot be NP-complete unless
$\text{NP}=\text{co-NP}$.
This  follows from our work below, but has been already proven by
Ramana \cite{ra97} in 1997; Porkolab and Khachiyan \cite{pk97} have proved that either for fixed number variables or
for fixed matrix size, the LMI feasibility problem lies in \text{P}.
Second, the standard (Lagrange-Slater) dual of a semidefinite program works well when the feasible set is full-dimensional
(e.g. if there is $x\in\R^n$ with $A(x)\succ0$). However, in general, strong duality can fail badly, and there is no easy
way of reducing to the full-dimensional case. 
Even the corresponding version of Farkas' lemma
fails for SDP.

\smallskip
We prove in this paper a nonlinear Farkas' lemma for SDP by giving
algebraic certificates for infeasibility of an LMI. Furthermore, we  
present a 
new exact duality theory for SDP.
The inspiration for our \emph{sums of squares dual} comes from real algebraic geometry,
more precisely from sums of squares
representations and the Real Nullstellensatz \cite{ma08, pd01, sc09}. We 
believe that this new connection will lead to further insights in the future.

\subsection*{Reader's guide}
The paper is organized as follows:
We fix terminology and notation in Section \ref{sec:not}.
Our main results, including the sums of squares dual of an SDP, are presented
in Section \ref{sec:sdp}. The two crucial ingredients needed in the proof are 
a low-dimensionality certificate for spectrahedra (see Subsection \ref{subs:lowdim}),
and a new Positivstellensatz for linear polynomials nonnegative on a spectrahedron (see 
Theorem \ref{indsos} in Subsection
\ref{subs:linpos}). 
Finally, in Subsection \ref{subs:newdual}, we present the sums of squares dual 
\hyperref[sdp:dsos]{${\rm (D^{\rm sos})} $}
of an SDP,
and Theorem \ref{sosram}, an algorithmic variant of the linear Positivstellensatz.
Section \ref{sec:pospoly} contains applications of these results.
For example, in Subsection \ref{subs:revisit} we 
interpret Theorems \ref{indsos} and \ref{sosram} in the language of real algebraic geometry,
and in 
Subsection \ref{subs:degbound} we prove 
a nonlinear Farkas' lemma for SDP by giving
nonlinear algebraic certificates for infeasibility of an LMI.
These results use quadratic modules from real algebraic geometry. 
As a side product we introduce a hierarchy for infeasibility of LMIs, whose first stage coincides with
 strong infeasibility. Subsection \ref{subs:bs} contains certificates for boundedness of
 spectrahedra which are used to give a Putinar-Schm\"udgen-like Positivstellensatz
 for polynomials positive on bounded spectrahedra. 
Finally, the article concludes with two brief sections containing examples illustrating our results, 
and an application to positive linear functionals.

\section{Notation and terminology}\label{sec:not}

We write $\N:=\{1,2,\dots\}$ and $\R$ for the sets of natural and real numbers, respectively.
Let $R$ be a unital commutative ring. For any matrix $A$ over $R$, we denote by $A^*$ its transpose.
Then $SR^{m\times m}:=\{A\in R^{m\times m}\mid A=A^*\}$ denotes the set of all \emph{symmetric} $m\times m$ matrices.
Examples of these include \emph{hermitian squares}, i.e.,
elements of the form $A^*A$ for some $A\in R^{n\times m}$.

Recall that a matrix $A\in\Rmm$ is called
\textit{positive semidefinite} (\emph{positive definite}) if it is
symmetric and $v^*Av\ge 0$ for all vectors $v\in\R^m$ ($v^*Av>0$ for all $v\in\R^m\setminus\{0\}$).
For real matrices $A,B\in\Rmm$, we write $A\preceq B$
(respectively $A\prec B$) to express that $B-A$ is positive semidefinite
(respectively positive definite). We denote by $S\R^{m\times m}_{\succeq0}$ and $S\R^{m\times m}_{\succ0}$ the convex cone
of all positive semidefinite and positive definite
matrices of size $m$, respectively.

\subsection{Matrix polynomials}

Let $\x=(\x_1,\dots,\x_n)$ be an $n$-tuple of 
commuting variables and $\R\cx$ the polynomial ring.
With $\R\cx_k$ we denote the vector space of all polynomials
of degree $\leq k$.
A (real) \emph{matrix polynomial} is a matrix whose entries are polynomials from $\RR\cx$.
It is \emph{linear} or \emph{quadratic} if its entries are from $\RR\cx_1$ or
$\RR\cx_2$, respectively.
A \emph{matrix polynomial} is an element of the ring $\R\cx^{m\times n}$ for some $m,n\in\N$, and can be viewed
either as a polynomial with matrix coefficients, or as a matrix whose entries are polynomials.
For a comprehensive treatment of the theory of matrix polynomials we refer the reader to the book \cite{glr82} and the references therein.

\subsection{Linear pencils and spectrahedra}

We use the term \emph{linear pencil} as a synonym and abbreviation for symmetric linear matrix polynomial.
A linear pencil $A\in\rx^{\al\times\al}$ can thus be written uniquely as $$A=A_0+\x_1A_1+\dots+\x_nA_n$$ with
$A_i\in S\R^{\al\times\al}$.
If $A\in\rx^{\al\times\al}$ is a linear pencil, then the condition $A(x)\succeq0$ ($x\in\R^n$) is called a \emph{linear matrix inequality} (LMI)
and its solution set $$S_A:=\{x\in\R^n\mid A(x)\succeq0\}$$
is called a \emph{spectrahedron} (or also an \emph{LMI set}).
We say that $A$ is \emph{infeasible} if $S_A=\emptyset$, and $A$ is \emph{feasible} if $S_A\neq\emptyset$.

Obviously, each spectrahedron is a closed convex semialgebraic subset of $\R^n$. Optimization of linear
objective functions over spectrahedra is called \emph{semidefinite programming} (SDP) \cite{bv96,to01,wsv00}. 
If $A\in\rx^{\al\times\al}$ is a \emph{diagonal} linear pencil, then $A(x)\succeq0$ ($x\in\R^n$) is just a (finite) system of (non-strict) linear inequalities
and $S_A$ is a (closed convex) polyhedron. Optimization of linear
objective functions over polyhedra is called \emph{linear programming} (LP).
The advantage of LMIs over systems of linear inequalities (or of spectrahedra over polyhedra, and SDP over LP, respectively) is a considerable gain of
expressiveness which makes LMIs an important tool in many areas of applied and pure mathematics \cite{befb94,go97,par00,las10}.
SDPs can be solved efficiently using interior point methods \cite{nn94,st00,dk02}.

\subsection{Sums of squares}

Another example of symmetric matrix polynomials that are of special interest to us are sums of hermitian
squares in $\RR\cx^{m\times m}$. They are called \emph{sos-matrices}. More explicitly,
$S\in\RR\cx^{m\times m}$ is an sos-matrix if the following equivalent
conditions hold:
\begin{enumerate}[\rm(i)]
\item $S=P^*P$ for some $s\in\N$ and some $P\in\RR\cx^{s\times m}$;
\item $S=\sum_{i=1}^rQ_i^*Q_i$ for some $r\in\N$ and $Q_i\in\RR\cx^{m\times m}$;
\item $S=\sum_{i=1}^sv_iv_i^*$ for some $s\in\N$ and $v_i\in\RR\cx^m$.
\end{enumerate}
A special case are sums of squares in the polynomial ring $\rx$. They are called sos-polynomials
and they are nothing else but sos-matrices of size $1$.
We denote the set of all sos-matrices (of any size) over $\rx$ by
$$\Si^2:=\left\{P^*P\mid s,m\in\N, P\in\RR\cx^{s\times m}\right\}.$$
In particular, $\Si^2\cap\rx$ is the set of sos-polynomials.

Note that an sos-matrix $S\in\RR\cx^{m\times m}$ is positive semidefinite on $\R^n$
but not vice-versa, since e.g.~a polynomial nonnegative on $\R^n$ is not necessarily sos \cite{ma08,pd01}.
 
\subsection{Radical ideals}

Recall that for any ideal $\cI\subseteq\rx$, its \emph{radical} $\sqrt\cI$ and its \emph{real radical} $\rrad\cI$ are the ideals defined by
\begin{align*}
\rad\cI&:=\{f\in\rx\mid \exists k\in\N:f^k\in\cI\}\qquad\text{and}\\
\rrad\cI&:=\{f\in\rx\mid \exists k\in\N:\exists s\in\Si^2\cap\rx:f^{2k}+s\in\cI\}.
\end{align*}
An ideal $\cI\subseteq\rx$ is called \emph{radical} if $\cI=\rad\cI$ and \emph{real radical} if $\cI=\rrad\cI$.
We refer the reader to \cite{bcr98} for further details.

\section{Duality theory of semidefinite programming}\label{sec:sdp}

In this section we present a sums of squares inspired dual for SDP,
see Subsection \ref{subs:newdual}. 
It is derived from  two core ingredients, which are of independent interest.
First, Proposition  \ref{exsos} below allows us to detect low-dimensionality
of spectrahedra, thus leading to a codimension-reduction technique.
Second, Theorem \ref{indsos} gives a nonlinear algebraic certificate (i.e., a Positivstellensatz)
for linear polynomials nonnegative on a spectrahedron.

\subsection{Weakly feasible and weakly infeasible linear pencils}\label{subs:weakly}

Recall that the linear pencil
$A({\tt x})\in\R\cx^{\al\times\al}_1$
is infeasible if $S_A=\emptyset$. We call $A$ \emph{strongly infeasible} if
$$\dist\big(\{A(x)\mid x\in\R^n\},\ S\R^{\al\times\al}_{\succeq0}\big)>0,$$ and 
\emph{weakly infeasible} if it is infeasible but is not strongly infeasible.
A feasible linear pencil $A$ is \emph{strongly feasible} if
there is an $x\in\R^n$ such that $A(x)\succ0$, and \emph{weakly feasible} otherwise.
To $A$ we associate the convex cone 
\begin{align*}
C_A&:=\Big\{c+\sum_iu_i^*Au_i\mid c\in\R\nn,\ u_i\in\R^\al\Big\}\\
&~=\Big\{c+\tr(AS)\mid c\in\R\nn,\ S\in S\R^{\al\times\al}_{\succeq0}\Big\}
\subseteq\R\ac_1.
\end{align*}
Note that $C_A$ consists of linear polynomials nonnegative on $S_A$.

The following is an extension of Farkas' lemma from LP to SDP due to Sturm \cite[Lemma 2.18]{st00}.
We include its simple proof based on a Hahn-Banach separation argument.

\begin{lemma}[Sturm]\label{sturm}
A linear pencil $A$ is strongly infeasible if and only if $-1\in C_A$.
\end{lemma}

\begin{proof}
Suppose
$$A=A_0+\sum_{i=1}^n {\tt x}_iA_i \in\R\cx^{\al\times\al}_1 $$
is strongly infeasible. Then the non-empty convex sets $\{A(x)\mid x\in\R^n\}$
and $\sapsd$ can be \emph{strictly}
separated by an affine hyperplane (since their Minkowski sums with a small ball are still disjoint and can therefore be separated
\cite[Theorem III.1.2]{ba02}).
This means that there is a non-zero linear functional $$\ell\colon\sa\to\R$$
and $\ga\in\R$,
with $\ell(\sapsd)\subseteq\R_{>\ga}$ and
$\ell(\{A(x)\mid x\in\R^n\})\subseteq\R_{<\ga}$.
Choose
$B\in\sa$ such that $$\ell(A)=\tr(AB)$$ for all $A\in\sa$.
Since $\ell(\sapsd)$ is bounded from below, by the self-duality of the
convex cone of positive semidefinite matrices, $0\neq B\succeq0$.
Similarly, we obtain $\ell(A_i)=0$
for $i\in\{1,\dots,n\}$. Note that $\ga<0$ since $0=\ell(0)\in\R_{>\ga}$ so we can assume
$\ell(A_0)=-1$ by scaling.
Writing $B=\sum_i u_iu_i^*$ with $u_i\in\R^\al$, we obtain
$$
-1= \ell(A_0)=\ell(A(x))= \tr (A(x)\sum_i u_iu_i^*)=\sum_i u_i^* A(x)u_i.
$$
for all $x\in\R^n$.
Hence $-1=\sum_i u_i^*Au_i\in C_A$.

Conversely, if $-1\in C_A$, i.e., $-1=c+\sum_i u_i^*Au_i$ for some $c\ge0$ and $u_i\in\R^\al$, then with $B:=\sum_iu_iu_i^*\in\sapsd$ we obtain a linear form
$$\ell\colon\sa\to\R, \quad A\mapsto \tr(AB)$$
satisfying $\ell(\sapsd)\subseteq\R_{\geq0}$ and $\ell(\{A(x)\mid x\in\R^n\})=\{-1-c\}\subseteq\R_{\le-1}$. So $A$ is strongly infeasible.
\end{proof}

\begin{lemma}\label{wind}
Let $A\in S\R\cx^{\al\times\al}_1$ be an infeasible linear pencil. Then the following are equivalent:
\begin{enumerate}[{\rm(i)}]
\item $A$ is weakly infeasible;\label{wind1}
\item $S_{A+\ep I_\al}\neq\emptyset$ for all $\ep>0$.\label{wind2}
\end{enumerate}
\end{lemma}

\begin{proof}
Since all norms on a finite-dimensional vector space are equivalent, we can without loss of generality use the operator norm on $\R^{\al\times\al}$.

Suppose that \eqref{wind1} holds and $\ep>0$ is given. Choose $B\in\sapsd$ and $x\in\R^n$ with $\|B-A(x)\|<\ep$.
Then $A(x)+\ep I_\al\succeq0$, i.e., $x\in S_{A+\ep I_m}$.

Conversely, suppose that \eqref{wind2} holds. To show that $$\dist\big(\{A(x)\mid x\in\R^n\},\ \sapsd\}\big)=0,$$ we let $\ep>0$ be given and have to find
$B\in\sapsd$ and $x\in\R^n$ with $$\|A(x)-B\|\le\ep.$$ 
But this is easy: choose $x\in\R^n$ with $A(x)+\ep I_\al\succeq0$, and set $B:=A(x)+\ep I_\al$.
\end{proof}

The following lemma is due to Bohnenblust \cite{bo48} (see also \cite[Theorem 4.2]{ba01} for an easier accessible reference). While Bohnenblust gave a non-trivial bound on the number of terms that are really
needed to test condition \eqref{bohne1} below, we will not need this improvement
and therefore take the trivial bound $\al$. 
Then the proof becomes easy and we include
it for the convenience of the reader.

\begin{lemma}[Bohnenblust]\label{bohne07}
For $A_1,\dots,A_n\in S\R^{\al\times \al}$ the following are equivalent:
\begin{enumerate}[\rm(i)]
\item Whenever $u_1,\dots,u_\al\in\R^\al$ with $\sum_{i=1}^\al u_i^*A_ju_i=0$ for all
$j\in\{1,\dots,n\}$, then $u_1=\dots=u_\al=0$;\label{bohne1}
\item $\Span(A_1,\dots,A_n)$ contains a positive definite matrix.\label{bohne2}
\end{enumerate}
\end{lemma}

\begin{proof}
It is trivial that \eqref{bohne2} implies \eqref{bohne1}. To prove that \eqref{bohne1} implies \eqref{bohne2}, note that
$\sapsd=\{\sum_{i=1}^\al u_iu_i^*\mid u_1,\dots,u_\al\in\R^\al\}$ and 
$\sum_{i=1}^\al u_i^*Bu_i=\tr(B\sum_{i=1}^\al u_iu_i^*)$ for all
$u_1,\dots,u_\al\in\R^\al$. The hypotheses thus says that, given any $B\in\sapsd$, we have 
\beq\label{eq:bbl1}
\tr(A_1B)=\dots=\tr(A_nB)=0\implies B=0.
\eeq
Now suppose that $\Span(A_1,\dots,A_n)\cap\sapd=\emptyset$. By the standard separation theorem for two non-empty disjoint convex sets
(see for example \cite[Theorem III.1.2]{ba02}), $\Span(A_1,\dots,A_n)$ and $\sapd$ can be separated by a hyperplane (the separating affine
hyperplane must obviously contain the origin). Therefore there is a non-zero linear functional $L\colon\sa\to\R$ with 
$$L(\sapd)\subseteq\R\nn\quad\text{and}\quad
L(\Span(A_1,\dots,A_n))\subseteq\R_{\le0}.$$ 
Then of course $L(\sapsd)\subseteq\R\nn$ and $L(\Span(A_1,\dots,A_n))=\{0\}$. Now choose
$B\in\sa$ such that 
$$L(A)=\tr(AB)\quad\text{for all}\quad A\in\sa.$$ 
Then $B\neq 0$, $B\in\sapsd$ and $\tr(A_1B)=\dots=\tr(A_nB)=0$,
contradicting \eqref{eq:bbl1}.
\end{proof}

\begin{lemma}\label{kaffeebohne07}
Let $A\in\R\cx^{\al\times\al}_1$ be a linear pencil which is either weakly infeasible or weakly feasible.
Then there are $k\ge1$ and $u_1,\dots,u_k\in\R^\al\setminus\{0\}$
such that $\sum_{i=1}^ku_i^*Au_i=0$.
\end{lemma}

\begin{proof}
Assume that the conclusion is false.
By Lemma \ref{bohne07}, we find  $x_0,x_1,\dots,x_n\in\R$ such that $$x_0A_0+x_1A_1+\dots+x_nA_n\succ0.$$
Of course it is impossible that $x_0>0$ since otherwise $A\big(\frac{x_1}{x_0},\ldots,\frac{x_n}{x_0} \big)\succ0$.
Also $x_0=0$ is excluded (since otherwise
$A(cx_1,\dots,cx_n)\succ0$ for $c>0$ large enough). Hence without loss of generality $x_0=-1$, i.e.,
$x_1A_1+\dots+x_nA_n\succ A_0$. Choose $\ep>0$ such that
$$x_1A_1+\dots+x_nA_n\succ A_0+2\ep I_\al.$$ By Lemma \ref{wind}, we can choose some $y\in S_{A+\ep I_\al}$. But then
$$A_0+(x_1+2y_1)A_1+\dots+(x_n+2y_n)A_n\succ2(A_0+\ep I_\al+y_1A_1+\dots+y_nA_n)\succeq0,$$ contradicting the
hypotheses.
\end{proof}

\subsection{An algebraic glimpse at standard SDP duality}\label{subs:standard}

We now recall briefly the standard duality theory of SDP. We present it from the viewpoint of a real algebraic geometer, i.e., we use the language
of polynomials in the formulation of the primal-dual pair of SDPs and in the proof of strong duality. This is necessary for a better understanding of
the sums of squares dual 
\hyperref[sdp:dsos]{${\rm (D^{\rm sos})} $}
given in Subsection \ref{subs:newdual} below.

A semidefinite program 
\hyperref[sdp:p]{(P)}
and its standard dual \hyperref[sdp:d]{(D)} is given by a linear pencil $A\in\R[\x]_1^{\al\times \al}$ and
a linear polynomial $\ell\in\R[\x]_1$ as follows:

\bigskip
$
\begin{array}[t]{lrl}
{\rm (P)}\label{sdp:p} &\text{minimize}&\ell(x)\\
&\text{subject to}&x\in\R^n\\
&&A(x)\succeq0
\end{array}
$\hfill
$
\begin{array}[t]{lrl}
{\rm (D)}\label{sdp:d} &\text{maximize}&a\\
&\text{subject to}&S\in S\R^{\al\times\al},\ a\in\R\\
&&S\succeq0\\
&&\ell-a=\tr(AS)
\end{array}
$

\bigskip\noindent
To see that this corresponds (up to some minor technicalities) to the formulation in the literature, just write the polynomial constraint
$\ell-a=\tr(AS)$ of the dual as $n+1$ linear equations by comparing coefficients.

The optimal values of \hyperref[sdp:p]{(P)} and \hyperref[sdp:d]{(D)} are defined to be
\begin{align*}
P^*&:=\inf\{\ell(x)\mid x\in\R^n,\ A(x)\succeq0\}\in\R\cup\{\pm\infty\}\qquad\text{and}\\
D^*&:=\sup\{a\mid S\in S\R^{\al\times\al}_{\succeq0},\ a\in\R,\ \ell-a=\tr(AS)\}\in\R\cup\{\pm\infty\},
\end{align*}
respectively, where the infimum and the supremum is taken in the ordered set $\{-\infty\}\cup\R\cup\{\infty\}$ (where $\inf\emptyset=\infty$ and
$\sup\emptyset=-\infty$).
By \emph{weak duality}, we mean that $P^*\ge D^*$, or equivalently, that the objective value of \hyperref[sdp:p]{(P)} at any of its feasible points is
greater or equal to the objective value of \hyperref[sdp:d]{(D)}  at any of its feasible points.

Fix a linear pencil $A$. It is easy to see that weak duality holds for all primal objectives $\ell$ if and only if
$$f\in C_A\implies\text{$f\ge0$ on $S_A$}$$
holds for all $f\in\rx_1$, which is of course true.
By \emph{strong duality}, we mean that $P^*=D^*$ (\emph{zero duality gap}) and that (the objective of) \hyperref[sdp:d]{(D)}  \emph{attains} this common optimal value
in case it is finite. It is a little exercise to see that strong duality for all primal objectives $\ell$ is equivalent to
$$\text{$f\ge0$ on $S_A$}\iff f\in C_A$$
for all $f\in\rx_1$.

Unlike weak duality, strong duality fails in general (cf.~Subsection \ref{subs:ex} below; Pataki recently characterized
all linear pencils $A$ such that there exists a linear objective function $\ell$ for which strong duality fails \cite{pat}).
However, it is
well-known that it does hold when the feasible set $S_A$ of the primal \hyperref[sdp:p]{(P)}  has non-empty interior
(e.g. if $A$ is strongly feasible). 
Here is a real algebraic geometry flavored proof of this:

\begin{prop}[Standard SDP duality]\label{usual}
Let $A\in S\rx^{\al\times\al}_1$ be a linear pencil such that $S_A$ has non-empty interior. Then
$$f\ge0\text{\ on\ }S_A\ \iff\ f\in C_A$$
for all $f\in\rx_1$.
\end{prop}

\begin{proof}
In a preliminary step, we show that the convex cone $C_A$ is closed in $\rx_1$. To this end,
consider the linear subspace $U:=\{u\in\R^\al\mid Au=0\}\subseteq\R^\al$.
The map
$$\ph\colon\R\times(\R^\al/U)^\al\to C_A,\quad(a,\bar u_1,\dots,\bar u_\al)\mapsto a^2+\sum_{i=1}^\al u_i^*Au_i$$
is well-defined and surjective.

Suppose $\ph$ maps $(a,\bar u_1,\dots,\bar u_\al)\in(\R^\al/U)^\al$ to $0$. Fix $i\in\{1,\dots,\al\}$. Then
$u_i^*A(x)u_i=0$ for all $x\in S_A$. Since $A(x)\succeq0$, this implies $A(x)u_i=0$ for all
$x\in S_A$. Using the hypothesis that $S_A$ has non-empty interior, we conclude that $Au_i=0$, i.e., $u_i\in U$.
Since $i$ was arbitrary and $a=0$, this yields $(a,\bar u_1,\dots,\bar u_\al)=0$.

This shows $\ph^{-1}(0)=\{0\}$. Together with the fact that $\ph$ is a (quadratically) homogeneous map, this implies that
$\ph$ is proper (see for example \cite[Lemma 2.7]{ps01}). In particular, $C_A=\im\ph$ is closed.

Suppose now that $f\notin\rx_1\setminus C_A$. The task is to find $x\in S_A$ such that $f(x)<0$.
Being a closed convex cone, $C_A$ is the intersection of all closed half-spaces containing
it. Therefore we find a linear map $\ps:\rx_1\to\R$ such that $\ps(C_A)\subseteq\R\nn$ and $\ps(f)<0$. We can assume
$\ps(1)>0$ since otherwise $\ps(1)=0$ and we can replace $\ps$ by $\ps+\ep\ev_y$ for some small $\ep>0$, where
$y\in S_A$ is chosen arbitrarily. Hereby $\ev_x\colon\rx_1\to\R$ denotes the evaluation in $x\in\R^n$. Finally, after a
suitable scaling we can even assume $\ps(1)=1$.

Now setting $x:=(\ps({\tt x}_1),\dots,\ps({\tt x}_n))\in\R^n$, we have $\ps=\ev_x$. So $\ps(C_A)\subseteq\R\nn$ means exactly that
$A(x)\succeq0$, i.e., $x\in S_A$. At the same time $f(x)=\ps(f)<0$ as desired.
\end{proof}

\subsection{Certificates for low dimensionality of spectrahedra}\label{subs:lowdim}

The problems with the standard duality theory for SDP thus arise when one deals with spectrahedra having empty interior.
Every convex set with empty interior is contained in an affine hyperplane. The basic idea is now to code the search for such an affine hyperplane
into the dual SDP and to replace equality in the constraint $\ell-a=\tr(AS)$ of \hyperref[sdp:d]{(D)}  by congruence modulo the linear
polynomial $f\in\R\cx_1$ defining the affine hyperplane. 
However, this raises  several issues:

First, $S_A$ might have codimension bigger than one in $\R^n$. This will be resolved by iterating the search up to $n$ times.

Second, we do not see any possibility to encode the search for the linear polynomial $f$ directly into an SDP. What we can implement is the search
for a non-zero \emph{quadratic} sos-polynomial $q$ together with a certificate of $S_A\subseteq\{q=0\}$. Note that $\{q=0\}$ is a proper affine subspace
of $\R^n$. It would be best to find a $q$ such that $\{q=0\}$ is the affine hull of $S_A$ since then we could actually avoid the $n$-fold iteration just
mentioned. However, as demonstrated in Example \ref{ex:new} below,  this is in general not possible.

Third, we need to carefully implement congruence modulo linear polynomials $f$ vanishing on $\{q=0\}$. This will be dealt with by
using the radical ideal from real algebraic geometry together with Schur complements.

\smallskip
We begin with a result which ensures that a suitable quadratic sos-polynomial $q$ can always be found. In fact, the following proposition says that
there exists such a $q$ which is actually a square. The statement is of interest in itself since it provides certificates for low-dimensionality of
spectrahedra. We need \emph{quadratic} (i.e., degree $\le 2$) sos-matrices for this.

\begin{proposition}\label{exsos}
For any linear pencil $A\in S\rx^{\al\times\al}_1$, the following are equivalent:
\begin{enumerate}[\rm(i)]
\item $S_A$ has empty interior;\label{exsos1}
\item There exists a non-zero linear polynomial\label{exsos2}
$f \in\rx_1$ and a quadratic sos-matrix $S\in S\rx^{\al\times\al}_2$ such that
\begin{equation}\label{eq:lowdim}
-f^2=\tr(AS).
\end{equation}
\end{enumerate}
\end{proposition}

\begin{proof}
From \eqref{exsos2} it follows that $-f^2\ge0$ on $S_A$, which implies $f=0$ on $S_A$.
So it is trivial that \eqref{exsos2} implies \eqref{exsos1}.

For the converse, suppose that $S_A$ has empty interior.
If there is $u\in\R^\al\setminus\{0\}$ such that $Au=0$ then, by an orthogonal change of
coordinates on $\R^\al$, we could assume that $u$ is the first standard basis vector $e_1$. But then we delete the first
column and the first row from $A$. We can iterate this and therefore assume from now on that there is no
$u\in\R^\al\setminus\{0\}$ with $Au=0$.

We first treat the easy case where $A$ is strongly infeasible. By Lemma~\ref{sturm}, there are
$c\in\R\nn$ and $u_i\in\R^\al$ with $-1-c=\sum_iu_i^*Au_i$. By scaling the $u_i$ we can assume $c=0$.
Setting $S:=\sum_iu_iu_i^*\in\sa$ and $f:=1$, we have $-f^2=-1=\sum_iu_i^*Au_i=\tr(AS)$ for the
constant sos-matrix $S$ and the constant non-zero linear polynomial $f$.

Now we assume that $A$ is weakly infeasible or feasible. In case that $A$ is feasible, it is clearly weakly feasible
since otherwise $S_A$ would have non-empty interior. Now Lemma \ref{kaffeebohne07} justifies the
following case distinction:

\smallskip
{\bf Case 1.} There is $u\in\R^\al\setminus\{0\}$ with $u^*Au=0$.

\noindent
Write $A=(\ell_{ij})_{1\le i,j\le \al}$. Again by an orthogonal change of coordinates on $\R^\al$, we can assume that $u=e_1$,
i.e., $\ell_{11}=0$. Moreover, we may assume $f:=\ell_{12}\neq0$ (since $Ae_1=Au\neq0$). Setting
$f':=\frac12(-1-\ell_{22})$,
$v:=[f'\ f\ 0\dots0]^*$ and $S:=vv^*$, we have
$$\tr(AS)=v^*Av=2 f'f\ell_{12}+f^2\ell_{22}=f^2(\ell_{22}+2f')=-f^2.$$

\smallskip
{\bf Case 2.} Case 1 does not apply but there are $k\ge 2$ and
$u_1,\dots,u_k\in\R^\al\setminus\{0\}$ such that $\sum_{i=1}^ku_i^*Au_i=0$.

\noindent
Here we set $f:=u_1^*Au_1\neq 0$ and write $-f=f_1^2-f_2^2$ where
$f_1:=\frac12(-f+1)\in\rxlin$ and $f_2:=\frac12(-f-1)\in\rxlin$.
Then we can use the quadratic sos-matrix
$S:=f_1^2u_1u_1^*+f_2^2\sum_{i=2}^ku_iu_i^*$ % =f_1^2u_1u_1^*-f_2^2u_1u_1^*=-f u_1u_1^*$$
%to get $\tr(AS)=-f u_1^*Lu_1=-f^2$.
to get
\[
\begin{split}
\tr(AS) & = \tr \Big( A \big( f_1^2u_1u_1^*+f_2^2\sum_{i=2}^ku_iu_i^* \big) \Big )  =
f_1^2 u_1^* A u_1 + f_2^2 \sum_{i=2}^k u_i^* A u_i \\
& = f_1^2 u_1^* A u_1-f_2^2u_1^* A u_1= (f_1^2-f_2^2)u_1^* A u_1=-f^2.\qedhere
\end{split}
\]
\end{proof}

The certificate \eqref{eq:lowdim} of low-dimensionality exists for \emph{some} but in general \emph{not for every} affine hyperplane containing the
spectrahedron. 
We illustrate this in  Example \ref{ex:new} below,
where the spectrahedron has codimension two and is therefore contained in
infinitely many affine hyperplanes \emph{only one} of which allows for a certificate of the form \eqref{eq:lowdim}.

\subsection{Linear polynomials positive on spectrahedra}\label{subs:linpos}

We now carry out the slightly technical but straightforward iteration of Proposition \ref{exsos} announced
in Subsection \ref{subs:lowdim}, and combine
it with Proposition \ref{usual}. We get a new type of Positivstellensatz for linear polynomials on spectrahedra with bounded degree complexity.
In what follows, we shall use $(p_1,\ldots,p_r)$ to denote the ideal generated by $p_1,\ldots, p_r$.

\begin{theorem}[Positivstellensatz for linear polynomials on spectrahedra]\label{indsos}\mbox{}\\
Let $A\in S\rx^{\al\times\al}_1$ be a linear pencil and $f\in\rx_1$. Then $$f\ge0\text{ on }S_A$$ if and only if  there exist
$\ell_1,\dots,\ell_n\in\rx_1$, quadratic sos-matrices $S_1,\dots,S_n\in S\rx^{\al\times\al}_2$, a matrix $S\in S\R^{\al\times\al}_{\succeq0}$
and $c\ge0$ such that
\begin{align}
\ell_i^2+\tr(AS_i)&\in(\ell_1,\dots,\ell_{i-1})\quad\text{for}\quad i\in\{1,\dots,n\},\qquad\text{and}\label{indsos1}\\
f-c-\tr(AS)&\in(\ell_1,\dots,\ell_n).\label{indsos2}
\end{align}
\end{theorem}

\begin{proof}
We first prove that $f\ge0$ on $S_A$ in the presence of \eqref{indsos1} and \eqref{indsos2}.

The traces in \eqref{indsos1} and \eqref{indsos2} are obviously nonnegative on $S_A$.
Hence it is clear that constraint \eqref{indsos2} gives $f\ge0$ on $S_A$ if we show that $\ell_i$ vanishes on $S_A$
for all $i\in\{1,\dots,n\}$. Fix $i\in\{1,\dots,n\}$ and assume by induction that $\ell_1,\dots,\ell_{i-1}$ vanish on $S_A$. Then
\eqref{indsos1} implies $\ell_i^2+\tr(AS_i)$ vanishes on $S_A$ and therefore also $\ell_i$.

Conversely, suppose now that $f\ge0$ on $S_A$. We will obtain the data with properties \eqref{indsos1} and \eqref{indsos2}
by induction on the number of variables $n\in\N_0$.

To do the induction basis, suppose first that $n=0$. Then $f\ge0$ on $S_A$ just means that the real number $f$ is
nonnegative if $A\in S\R^{\al\times\al}$ is positive semidefinite. But if $f\ge0$, then it suffices to choose $c:=f\ge0$ and $S:=0$
to obtain \eqref{indsos2} with $n=0$, and the  condition \eqref{indsos1} is empty since $n=0$.
We now assume that $f<0$ and therefore $A\not\succeq0$. Then we choose $u\in\R^\al$ with
$u^*Au=f$. Setting $S:=uu^*\in S\R^{\al\times\al}_{\succeq0}$ and $c:=0$, we have $$f-c-\tr(AS)=f-u^*Au=f-f=0,$$ as required.

For the induction step, we now suppose that $n\in\N$ and that we know already how to find the required data for linear
pencils in $n-1$ variables. We distinguish two cases and will use the induction hypothesis only in the second one.

{\bf Case 1.} $S_A$ contains an interior point.

\noindent
In this case, we set all $\ell_i$ and $S_i$ to zero so that \eqref{indsos1} is trivially satisfied. Property
\eqref{indsos2} can be fulfilled by Proposition \ref{usual}.

{\bf Case 2.} The interior of $S_A$ is empty.

\noindent
In this case, we apply
Proposition~\ref{exsos} to obtain $0\neq\ell_1\in\rx_1$ and a quadratic sos-matrix $S_1\in S\rx^{\al\times \al}$ with
\begin{equation}\label{anfang}
\ell_1^2+\tr(AS_1)=0.
\end{equation}
The case where $\ell_1$ is constant is trivial. In fact, in this case we can choose all remaining data being zero since
$(\ell_1,\dots,\ell_i)=(\ell_1)=\rx$ for all $i\in\{1,\dots,n\}$.

From now on we therefore assume $\ell_1$ to be non-constant. But then the reader easily checks that there is no harm
carrying out an affine linear variable transformation which allows us to assume $\ell_1={\tt x}_n$.
We then apply the induction hypothesis to the linear pencil $A':=A({\tt x}_1,\dots,{\tt x}_{n-1},0)$ and the linear polynomial
$f':=f({\tt x}_1,\dots,{\tt x}_{n-1},0)$ in $n-1$ variables to obtain
$\ell_2,\dots,\ell_n\in\rx_1$, quadratic sos-matrices $S_2,\dots,S_n\in S\rx^{\al\times \al}_2$, a matrix $S\in \sapsd$ and 
a constant $c\ge0$ such that
\begin{align}
\ell_i^2+\tr(A'S_i)&\in(\ell_2,\dots,\ell_{i-1})\quad\text{for}\quad i\in\{2,\dots,n\}\qquad\text{and}\label{isos1}\\
f'-c-\tr(A'S)&\in(\ell_2,\dots,\ell_n).\label{isos2}
\end{align}
Noting that both $f-f'$ and $\tr(AS_i)-\tr(A'S_i)=\tr((A-A')S_i)$ are contained in the ideal $({\tt x}_n)=(\ell_1)$,
we see that \eqref{isos1} together with \eqref{anfang}
implies \eqref{indsos1}. In the same manner, \eqref{isos2} yields \eqref{indsos2}.
\end{proof}

\subsection{Constructing  SDPs for sums of squares problems}\label{subs:gram}

The (coefficient tuples of) sos-polynomials in $\rx$ of bounded degree 
form \emph{a projection of} a spectrahedron.
In other words, the condition of being (the coefficient tuple of) an sos-polynomial in $\rx$ of bounded degree can be expressed with an LMI
by means of additional variables. This is the well-known \emph{Gram matrix method} \cite{lau09,ma08}. As noted by Kojima \cite{ko03} and nicely
described by Hol and Scherer \cite{sh06}, the Gram matrix method extends easily to sos-\emph{matrices} (see
Example \eqref{wesuck} below).

\subsection{Real radical computations}\label{subs:realrad}

Let $A\in S\rx^{\al\times \al}_1$ be a linear pencil and $q\in\rx_2$ a (quadratic) sos-polynomial such that
$-q=\tr(AS)$ for some (quadratic) sos-matrix $S$ like in \eqref{eq:lowdim} above. In order to resolve the third issue mentioned
in Subsection \ref{subs:lowdim}, we would like to get our hands on (cubic) polynomials vanishing on $\{q=0\}$. That is, we want to
implement the ideals appearing in \eqref{indsos1} and \eqref{indsos2} in an SDP.

By the Real Nullstellensatz \cite{bcr98,ma08,pd01}, each polynomial vanishing on the real zero set $\{q=0\}$ of $q$
lies in $\rrad{(q)}$. This gives us a strategy of how to find the cubic polynomials vanishing on
$\{q=0\}$, cf. Proposition \ref{matrad} and Lemma \ref{vepaco} below. The Real Nullstellensatz plays only a motivating role for us;
we only use its trivial converse: Each element of $\rrad{(q)}$ vanishes on $\{q=0\}$.

The central question is  how to model the search for elements in the real radical ideal using SDP. The key to this will be to represent polynomials
by matrices as is done in the Gram matrix method mentioned in Section 
\ref{subs:gram}. For this we  introduce some notation.

For each $d\in\N_0$, let $\s d:=\dim\rx_d=\binom{d+n}n$ denote the number of monomials of degree at most $d$ in $n$ variables
and $\vecx d\in\rx^{\s d}$ the column vector
$$\vecx d:=\begin{bmatrix}1&{\tt x}_1&{\tt x}_2&\dots&{\tt x}_n&{\tt x}_1^2&{\tt x}_1{\tt x}_2&\dots&\dots&{\tt x}_n^d\end{bmatrix}^*$$
consisting of these monomials ordered first with respect to the degree and then lexicographic.

The following proposition shows how to find elements of degree at most $d+e$ (represented by a matrix $W$) in the real radical
$\cI:=\rrad{(q)}$ of the ideal generated by a polynomial $q\in\rx_{2d}$
(represented by a symmetric matrix $U$, i.e., $q=\vecx d^*U\vecx d$). We will later use it with $d=1$ and $e=2$ since $q$ will be
quadratic and we will be interested in  cubic polynomials in $\cI$.
Note that $$U\succeq W^*W\iff\begin{bmatrix}I&W\\W^*&U\end{bmatrix}\succeq0$$
by the method of Schur complements.

\begin{proposition}\label{matrad}
Let $d,e\in\N_0$,  let $\cI$ be a real radical ideal of $\rx$ and $U\in S\R^{s(d)\times s(d)}$ be such that
$\vecx d^*U\vecx d\in\cI$.  Suppose $W\in\R^{s(e)\times s(d)}$ with $U\succeq W^*W$. Then $\vecx e^*W\vecx d\in\cI$.
\end{proposition}

\begin{proof}
Since $U-W^*W$ is positive semidefinite, we find $B\in\R^{s(d)\times s(d)}$ with
$U-W^*W=B^*B$.
Now let $p_i\in\rx$ denote the $i$-th entry of $W\vecx d$ and $q_j$ the $j$-th entry of $B\vecx d$.
From
\begin{align*}
p_1^2+\dots+p_{s(e)}^2+q_1^2+\dots+q_{s(d)}^2&=(W\vecx d)^*W\vecx d+(B\vecx d)^*B\vecx d\\
&=\vecx d^*(W^*W+B^*B)\vecx d=\vecx d^*U\vecx d\in\cI
\end{align*}
it follows that $p_1,\dots,p_{s(e)}\in\cI$ since $\cI$ is real radical. Now
$$\vecx e^*W\vecx d=\vecx e^*[p_1\dots p_{s(e)}]^*=[p_1\dots p_{s(e)}]\vecx e\in\cI$$
since $\cI$ is an ideal.
\end{proof}

The following lemma is a weak converse to Proposition~\ref{matrad}. Its proof relies heavily
on the fact that only linear and quadratic polynomials are involved.

\begin{lemma}\label{vepaco}
Let $\ell_1,\dots,\ell_t\in\rxlin$, and $q_1,\dots,q_t\in\rx_2$. Suppose that $U\in S\R^{s(1)\times s(1)}$  satisfies
\beq\label{eqU}
\vecx 1^*U\vecx 1=\ell_1^2+\dots+\ell_t^2.
\eeq
Then there exists $\la>0$ and $W\in\R^{s(2)\times s(1)}$ satisfying $\la U\succeq W^*W$ and
$$\vecx 2^*W\vecx 1=\ell_1q_1+\dots+\ell_t q_t.$$
\end{lemma}

\begin{proof}
Note that the $U$ satisfying \eqref{eqU} is unique and hence positive semidefinite.
Suppose that at least one $q_i\neq0$ (otherwise take $W=0$).
Choose column vectors $c_i\in\R^{s(2)}$ such that $c_i^*\vecx2=q_i$.
Now let $W\in\R^{s(2)\times s(1)}$ be the matrix defined by $W\vecx1=\sum_{i=1}^t\ell_ic_i$, so that
$$\vecx2^*W\vecx1=\sum_{i=1}^t\ell_i\vecx2^*c_i=\sum_{i=1}^t\ell_ic_i^*\vecx2=
\sum_{i=1}^t\ell_iq_i.$$
Moreover, we get
$$\vecx1^*W^*W\vecx1=(W\vecx1)^*W\vecx1=\sum_{i,j=1}^t(\ell_ic_i)^*(\ell_jc_j),$$
and therefore for all $x\in\R^n$,
\begin{align*}
\begin{bmatrix}1&x_1&\dots&x_n\end{bmatrix}W^*W\begin{bmatrix}1\\x_1\\\vdots\\x_n\end{bmatrix}
&=\sum_{i,j=1}^t(\ell_i(x)c_i)^*(\ell_j(x)c_j)\\
&\le\frac12\sum_{i,j=1}^t \big((\ell_i(x)c_i)^*(\ell_i(x)c_i)+(\ell_j(x)c_j)^*(\ell_j(x)c_j)\big)\\
&= t\sum_{i=1}^t(\ell_i(x)c_i)^*(\ell_i(x)c_i)\le\la\sum_{i=1}^t\ell_i(x)^2,
\end{align*}
where we set $\la:=t\sum_{i=1}^tc_i^*c_i>0$. Therefore
$$\begin{bmatrix}1&x_1&\dots&x_n\end{bmatrix}(\la U-W^*W)\begin{bmatrix}1&x_1&\dots&x_n\end{bmatrix}^*\ge0$$
for all $x\in\R^n$. By homogeneity and continuity this implies $y^*(\la U-W^*W)y\ge0$ for all
$y\in\R^{s(1)}$, i.e., $\la U\succeq W^*W$.
\end{proof}

\subsection{A new exact duality theory for SDP}\label{subs:newdual}

Given an SDP of the form \hyperref[sdp:p]{(P)} described in Subsection~\ref{subs:standard}, the following is what we call its
\emph{sums of squares dual}:
$$
\begin{array}[t]{rll}
{\rm (D^{\rm sos})} \quad \label{sdp:dsos} 
\text{maximize}&a\\
\text{subject to}&S\in\sapsd, 
\ a\in\R\\
&S_1,\dots,S_n\in S\rx^{\al\times\al}_2\text{ quadratic sos-matrices\!}\!\\
&U_1,\dots,U_n\in S\R^{s(1)\times s(1)}\\
&W_1,\dots,W_n\in\R^{s(2)\times s(1)}\\
&\vecx1^*U_i\vecx1+\vecx2^*W_{i-1}\vecx1+\tr(AS_i)=0&\!\!(i\in\{1,\dots,n\})\\
&U_i\succeq W_i^*W_i&\!\!(i\in\{1,\dots,n\})\\
&\ell-a+\vecx2^*W_n\vecx1-\tr(AS)=0,
\end{array}
$$
where $W_0:=0\in\R^{s(2)\times s(1)}$.

\begin{remark}
Just like Ramana's extended Lagrange-Slater dual \cite{ra97},
\hyperref[sdp:dsos]{${\rm (D^{\rm sos})} $}
 can be written down in polynomial time (and hence has polynomial size)
in the bit size of the primal (assuming the latter has rational coefficients) and it guarantees that strong duality (i.e., weak duality,
zero gap and dual attainment) always holds.
Similarly, the facial reduction \cite{bw81,tw} gives rise to a good duality theory of SDP. We refer the reader to
\cite{Pat} for a unified treatment of these two constructions.
\end{remark}

As mentioned in Section \ref{subs:gram}, the quadratic sos-matrices can easily be modeled by SDP constraints using the
Gram matrix method, and the polynomial identities can be written as linear equations by comparing coefficients.
The $S_i$ serve to produce negated quadratic sos-polynomials vanishing on $S_A$ (cf. Proposition \ref{exsos}) which are captured by the
matrices $U_i$. From this, cubics vanishing on $S_A$ are produced (cf. Subsection \ref{subs:realrad}) and represented by the matrices
$W_i$. These cubics serve to implement the congruence modulo the ideals from \eqref{indsos1} and \eqref{indsos2}. Then the entire procedure
is iterated $n$ times.
We present an explicit example in Section \ref{subs:ex}.

Just as Proposition \ref{usual} corresponds to the standard SDP duality, Theorem \ref{sosram} 
below translates into the strong duality for the sums of squares
dual \hyperref[sdp:dsos]{${\rm (D^{\rm sos})} $}. Before we come to it, we need a folk lemma,  well-known from the theory of Gröbner bases.

\begin{lemma}\label{folk}
Suppose $d\in\N$, $f\in\rx_d$ and $\ell_1,\dots,\ell_t\in\rx_1$ are linear polynomials such that
$f\in(\ell_1,\dots,\ell_t)$. Then at least one of the following is true:
\begin{enumerate}[\rm(a)]
\item there exist $p_1,\dots,p_t\in\rx_{d-1}$ such that $f=p_1\ell_1+\dots+p_t\ell_t$;
\item there are $\la_1,\dots,\la_t\in\R$ such that $\la_1\ell_1+\dots+\la_t\ell_t=1$.
\end{enumerate}
\end{lemma}

\begin{proof} 
Suppose that (b) is not fulfilled. Then we may assume by Gaussian elimination and after renumbering
the variables that
$\ell_i={\tt x}_i-\ell_i'$ where $\ell_i'\in\R[{\tt x}_{i+1},\dots,{\tt x}_n]_1$. 
We now proceed by induction on $t\in\N_0$ to prove (a).
For $t=0$, there is nothing to show.
Now let $t\in\N$ and suppose the lemma is already proved with $t$ replaced by $t-1$.
Write
$f=\sum_{|\al|\le d}a_\al\x^\al$ with $a_\al\in\R$. Setting $g:=f(\ell_1',{\tt x}_2,\dots,{\tt x}_n)$, we have
$$f-g=\sum_{\genfrac{}{}{0pt}{1}{|\al|\le d}{1\le\al_1}}
a_\al({\tt x}_1^{\al_1}-\ell_1^{\prime\al_1}){\tt x}_2^{\al_2}\dotsm {\tt x}_n^{\al_n}=p_1({\tt x}_1-\ell_1')=p_1\ell_1,$$
where 
\[p_1:=\sum_{\genfrac{}{}{0pt}{1}{|\al|\le d}{1\le\al_1}}
a_\al\left(\sum_{i=0}^{\al_1-1}{\tt x}_1^i\ell_1^{\prime\al_1-1-i}\right){\tt x}_2^{\al_2}\dotsm {\tt x}_n^{\al_n}\in\rx_{d-1}.\]
Moreover, $g\in(\ell_2,\dots,\ell_t)$ and therefore $g=p_2\ell_2+\dots+p_t\ell_t$ for some $p_2,\dots,p_t\in\rx_{d-1}$ by
the induction hypothesis. Now 
\[
f=(f-g)+g=p_1\ell_1+\dots+p_t\ell_t.
\qedhere\]
\end{proof}

\begin{theorem}[Sums of squares SDP duality]\label{sosram}
Let $A\in S\rx^{\al\times\al}_1$ be a linear pencil and $f\in\rx_1$. Then $$f\ge0\text{ on }S_A$$ if and only if there exist
quadratic sos-matrices
$S_1,\dots,S_n\in S\rx^{\al\times \al}_2$, matrices $U_1,\dots,U_n\in S\R^{s(1)\times s(1)}$, $W_1,\dots,W_n\in\R^{s(2)\times s(1)}$,
$S\in\sapsd$ and $c\in\R_{\ge0}$ such that
\begin{align}
&\vecx1^*U_i\vecx1+\vecx2^*W_{i-1}\vecx1+\tr(AS_i)=0&(i\in\{1,\dots,n\}),\label{sosram1}\\
&U_i\succeq W_i^*W_i&(i\in\{1,\dots,n\}),\label{sosram2}\\
&f-c+\vecx2^*W_n\vecx1-\tr(AS)=0,\label{sosram3}
\end{align}
where $W_0:=0\in\R^{s(2)\times s(1)}$.
\end{theorem}

\begin{proof} We first prove that existence of the above data implies $f\ge0$ on $S_A$.
All we will use about the traces appearing in \eqref{sosram1} and \eqref{sosram3} is that they are polynomials nonnegative on $S_A$.
Let $\cI$ denote the real radical ideal of all polynomials vanishing on $S_A$.
It is clear that \eqref{sosram3} gives $f\ge0$ on $S_A$ if we show that
$\vecx2^*W_n\vecx1\in\cI$. In fact, we prove by induction that $\vecx2 ^*W_i\vecx1\in\cI$
for all $i\in\{0,\dots,n\}$.

The case $i=0$ is trivial since $W_0=0$ by definition. Let $i\in\{1,\dots,n\}$ be given and suppose that $\vecx2^*W_{i-1}\vecx1\in\cI$.
Then \eqref{sosram1} shows $\vecx1^*U_i\vecx1\le0$ on $S_A$. On the other hand, \eqref{sosram2} implies in particular
$U_i\succeq0$ and therefore $\vecx1^*U_i\vecx1\ge0$ on $S_A$. Combining both, $\vecx1^*U_i\vecx1\in\cI$.
Now Proposition \ref{matrad}  implies $\vecx2^*W_i\vecx1\in\cI$ by \eqref{sosram2}.
This ends the induction and shows $f\ge0$ on $S_A$ as claimed.

Conversely, suppose now that $f\geq0$ on $S_A$. By Theorem \ref{indsos} and Lemma \ref{folk}, we can choose
$\ell_1,\dots,\ell_n\in\rx_1$, quadratic sos-matrices $S_1',\dots,S_n'\in S\rx^{\al\times\al}$, $S\in\sapsd$ and
$q_{ij}\in\rx_2$ ($1\le j\le i\le n$) such that
\begin{align}
\ell_1^2+\dots+\ell_i^2+\tr(AS_i')&=\sum_{j=1}^{i-1}q_{(i-1)j}\ell_j&(i\in\{1,\dots,n\})\qquad\text{and}
\label{prep1}\\
f-c-\tr(AS)&=\sum_{j=1}^nq_{nj}\ell_j.\label{prep2}
\end{align}
There are two little arguments involved in this: First, \eqref{indsos1} can trivially be rewritten as
$\ell_1^2+\dots+\ell_i^2+\tr(AS_i')\in(\ell_1,\dots,\ell_{i-1})$ for $i\in\{1,\dots,n\}$. Second, in Lemma \ref{folk} 
applied to
$\ell_1,\dots,\ell_{i-1}$
($i\in\{1,\dots,n+1\}$)
we might fall into case (b).
But  then we may set
$\ell_{i}=\dots=\ell_n=0$ and $S_{i}'=\dots=S_n'=S=0$.

Now define $U_i'\in S\R^{s(1)\times s(1)}$ by $\vecx1^*U_i'\vecx1=\ell_1^2+\dots+\ell_i^2$ for $i\in\{1,\dots,n\}$. Using Lemma \ref{vepaco},
we can then choose $\la>0$ and $W_1',\dots,W_n'\in\R^{s(2)\times s(1)}$ such that
\begin{equation}\label{cop}
\la U_i'\succeq W_i'^*W_i'
\end{equation}
and
$\vecx2^*W_i'\vecx1=-\sum_{j=1}^iq_{ij}\ell_j$ for $i\in\{1,\dots,n\}$. Setting $W_n:=W_n'$, equation \eqref{prep2}
becomes \eqref{sosram3}. Moreover, \eqref{prep1} can be rewritten as
\begin{equation}\label{rew}
\vecx1^*U_i'\vecx1+\vecx2^*W_{i-1}'\vecx1+\tr(AS_i')=0\qquad(i\in\{1,\dots,n\})
\end{equation}
To cope with the problem that $\la$ might be larger than $1$ in \eqref{cop}, we look for $\la_1,\dots,\la_n\in\R\pos$ such
that $U_i\succeq W_i^*W_i$ for all $i\in\{1,\dots,n\}$ if we define $U_i:=\la_iU_i'$ and $W_{i-1}:=\la_iW_{i-1}'$ for all
$i\in\{1,\dots,n\}$ (in particular $W_0=W_0'=0$). With this choice, the desired linear matrix inequality \eqref{sosram2} is now
equivalent to $\la_iU_i'\succeq\la_{i+1}^2W_i'^*W_i'$ for $i\in\{1,\dots,n-1\}$ and $\la_nU_n'\succeq W_n'^2$.
Looking at \eqref{cop}, we therefore see that any choice of the $\la_i$ satisfying $\la_i\ge\la\la_{i+1}^2$ for
$i\in\{1,\dots,n-1\}$ and $\la_n\ge\la$ ensures \eqref{sosram2}. Such a choice is clearly possible. Finally, equation
\eqref{rew} multiplied by $\la_i$ yields \eqref{sosram1} by setting $S_i:=\la_iS_i'$ for $i\in\{1,\dots,n\}$.
\end{proof}

\section{Positivity of polynomials on spectrahedra}\label{sec:pospoly}

In this section we present applications of the results presented in Section \ref{sec:sdp}.
We 
interpret Theorems \ref{indsos} and \ref{sosram} in the language of real algebraic geometry
in Subsection \ref{subs:revisit}, and 
prove 
a nonlinear Farkas' lemma for SDP, i.e., 
nonlinear algebraic certificates for infeasibility of an LMI,
in Subsection \ref{subs:degbound}.
These results use quadratic modules from real algebraic geometry, which we recall
in Subsection \ref{subs:qm}.
As a side product we obtain a hierarchy for infeasibility of LMIs, whose first stage coincides with
 strong infeasibility. Subsection \ref{subs:bs} contains certificates for boundedness of
 spectrahedra and a Putinar-Schm\"udgen-like Positivstellensatz
 for polynomials positive on bounded spectrahedra. 
Finally, the section concludes with two brief subsections containing examples illustrating our results, 
and an application to positive linear functionals.

\subsection{Quadratic module
associated to a linear pencil}\label{subs:qm}

Let $R$ be a (commutative unital) ring. 
A subset $M\subseteq R$ is called a \emph{quadratic module} in $R$ if it
contains $1$ and is closed under addition and multiplication with squares, i.e.,
$$1\in M,\quad M+M\subseteq M,\quad\text{and}\quad a^2M\subseteq M\text{\ for all $a\in R$},$$
see for example \cite{ma08}. The \emph{support} of $M$ is defined to be $\supp M:=M\cap (-M)$.
A quadratic module $M\subseteq R$ is called \emph{proper} if $-1\not\in M$. 
If $\frac12\in R$, then the identity
\begin{equation}\label{id4}
4a=(a+1)^2-(a-1)^2\qquad\text{for all $a\in R$},
\end{equation}
shows that $\supp M$ is an ideal and therefore any improper quadratic module $M$ equals $R$ (since $1$ is contained
in its support).

An LMI $A(x)\succeq0$ can be seen as the infinite family of simultaneous linear inequalities $u^*A(x)u\ge0$ ($u\in\R^\al$).
In optimization,
when dealing with families of linear inequalities, one 
often considers the convex cone generated by them
(cf.~$C_A$ in Subsection \ref{subs:weakly}). Real algebraic geometry handles 
\emph{arbitrary} polynomial inequalities and
uses the multiplicative structure of the polynomial ring. Thence
one considers more special types of convex cones, like 
quadratic modules.
One of the aims of this section is to show that it is advantageous to consider
quadratic modules for the study of LMIs. Since quadratic
modules in polynomial rings are infinite-dimensional convex cones, we will later on also consider certain finite-dimensional truncations of them, see Subsection
\ref{subs:degbound}.

\begin{definition}\label{def:scm}
Let $A$ be a linear pencil of size $\al$ in the variables $\x$. We introduce
\begin{align*}
M_A&:=\Big\{s+\sum_iv_i^*Av_i\mid s\in\sos,\ v_i\in\rx^\al\Big\}\\
&~=\Big\{s+\tr(AS)\mid s\in\R\cx\text{ sos-polynomial},\ S\in\R\cx^{\al\times\al}\text{ sos-matrix}\Big\}
\subseteq\R\ac,
\end{align*}
and call it the \emph{quadratic module}
associated to the linear pencil $A$.
\end{definition}

Note that for any linear pencil $A$, each element of $M_A$ is a polynomial nonnegative on the spectrahedron $S_A$.
In general, $M_A$ does not contain all polynomials nonnegative on $S_A$ 
(e.g. when $A$ is diagonal and
$\dim S_A\ge 3$ \cite{sc09};
another simple example is presented in Example \ref{new2} below).
For diagonal $A$, the quadratic module $M_A$ (and actually the convex cone $C_A$) contains however all \emph{linear} polynomials
nonnegative on the polyhedron $S_A$ by Farkas' lemma.
For non-diagonal linear pencils
$A$ even this can fail, see Example \ref{new2} below.
To certify nonnegativity of a linear polynomial on a spectrahedron, we therefore employ more involved algebraic certificates
(motivated by the real radical), cf.~Theorem \ref{indsos} and its SDP-implementable version Theorem \ref{sosram}.
These two theorems
yield algebraic certificates for
linear polynomials nonnegative on $S_A$. 
While they  have the advantage of being very well-behaved with respect to complexity
issues, their statements are somewhat technical. Leaving complexity issues aside, one can use 
them to deduce a cleaner algebraic characterization of linear polynomials nonnegative on $S_A$.
Indeed,
given $f\in\rx_1$ with $f\ge0$ on $S_A$, the certificates in Theorems \ref{indsos} and \ref{sosram} 
can be seen to be equivalent to  $f\in C_A+\rad{\supp M_A}$ by means of Prestel's theory of semiorderings.
Note that each element of $C_A+\rad{\supp M_A}$ is obviously nonnegative on $S_A$ since the elements of $\rad{\supp M_A}$ vanish on $S_A$.

Finally, this will allow us to come back to the quadratic module $M_A$. We will show that it contains each linear polynomial nonnegative
on $S_A$ \emph{after adding an arbitrarily small positive constant}, see Corollary \ref{cor:almostsat}.

\smallskip
In this subsection, basic familiarity with real algebraic geometry as presented e.g.~in \cite{bcr98,ma08,pd01} is needed.
The following proposition follows easily from Prestel's theory of semiorderings on a commutative ring, see for example \cite[1.4.6.1]{sc09}.

\begin{proposition}\label{prestel}
Let $M$ be a quadratic module in $\rx$. Then
$$\rad{\supp M}=\rrad{\supp M}=\bigcap\{\supp S\mid \text{$S$ semiordering of $\rx$}, M\subseteq S\}.$$
\end{proposition}

We explicitly extract the following consequence since this is exactly what is needed in the sequel.

\begin{lemma}\label{lem:this}
Let $M$ be a quadratic module in $\rx$. Then
\begin{equation}
(\rad{\supp M}-M)\cap M\subseteq\rad{\supp M}.\label{prestelcor1}
\end{equation}
\end{lemma}

\begin{proof}
To prove \eqref{prestelcor1},
suppose $p\in M$ can be written $p=g-q$ with $g\in\rad{\supp M}$ and $q\in M$. By Proposition \ref{prestel}, we have
to show that $p\in\supp S$ for each semiordering $S$ of $\rx$ with $M\subseteq S$. But if such $S$ is given, then
$g\in\supp S$ and therefore $p=g-q\in -S$ as well as $p\in M\subseteq S$. Hence $p\in\supp S$.
\end{proof}

\subsection{Linear polynomials positive on spectrahedra -- revisited}\label{subs:revisit}

With Lemma \ref{lem:this} at hand, we can now give a conceptual interpretation of the certificates appearing in Theorem \ref{indsos}, disregarding the 
complexity of the certificate.

\begin{proposition}\label{indsos++}
If $A\in S\rx^{\al\times\al}_1$ is a linear pencil,
$f,\ell_1,\dots,\ell_n\in\rx_1$ are linear polynomials, $S_1,\dots,S_n\in S\rx^{\al\times\al}_2$ are quadratic sos-matrices, $S\in\sapsd$
and $c\in\R_{\ge0}$ are such that \eqref{indsos1} and \eqref{indsos2} hold, then $f\in C_A+\rad{\supp M_A}$.
\end{proposition}

\begin{proof}
Set $\cI:=\rad{\supp M_A}$.
It is clear that \eqref{indsos2} gives $f\in C_A+\cI$ if we prove that
$\ell_i\in\cI$ for all $i\in\{1,\dots,n\}$. Fix $i\in\{1,\dots,n\}$ and assume by induction that $\ell_1,\dots,\ell_{i-1}\in\cI$. Then
\eqref{indsos1} implies $\ell_i^2+\tr(AS_i)\in\cI$ and therefore $\ell_i^2\in(\cI-M_A)\cap\sos\subseteq(\cI-M_A)\cap M_A
\subseteq\cI$ by \eqref{prestelcor1}.
\end{proof}

We get the same interpretation for the certificates from Theorem \ref{sosram}.

\begin{proposition}\label{sosram++}
If $A\in S\rx^{\al\times\al}_1$ is a linear pencil,
$f\in\rx_1$ is a linear polynomial, $S_1,\dots,S_n\in S\rx^{\al\times\al}_2$ are quadratic sos-matrices,
$U_1,\dots,U_n\in S\R^{s(1)\times s(1)}$, $W_1,\dots,W_n\in\R^{s(2)\times s(1)}$, $S\in S\R^{\al\times \al}_{\succeq0}$ and $c\in\R_{\ge0}$ are such that
\eqref{sosram1}, \eqref{sosram2} and \eqref{sosram3} hold, then $f\in C_A+\rad{\supp M_A}$.
\end{proposition}

\begin{proof}
Set $\cI:=\rad{\supp M_A}$.
It is clear that constraint \eqref{sosram3} gives $f\in C_A+\cI$ if we show that
$\vecx2W_n\vecx1\in\cI$. In fact, we show by induction that $\vecx2W_i\vecx1\in\cI$ for all $i\in\{0,\dots,n\}$.

The case $i=0$ is trivial since $W_0=0$ by definition. Let $i\in\{1,\dots,n\}$ be given and suppose
that we know already $\vecx2^*W_{i-1}\vecx1\in\cI$.
Then \eqref{sosram1} shows $\vecx1^*U_i\vecx1\in\cI-M_A$. On the other hand \eqref{sosram2} implies in particular
$U_i\succeq0$ and therefore $\vecx1^*U_i\vecx1\in\sos\subseteq M_A$. But then
$\vecx1^*U_i\vecx1\in(\cI-M_A)\cap M_A\subseteq\cI$ by \eqref{prestelcor1}. Now \eqref{sosram2} yields
$\vecx2^*W_i\vecx1\in\cI$ by Proposition \ref{matrad} since $\cI$ is real radical by Proposition \ref{prestel}.
This ends the induction.
\end{proof}

The following corollary is now a generalization of Proposition \ref{usual} working also for low-dimensional $S_A$ (note that $\supp M_A=(0)$ if
$S_A$ has non-empty interior).

\begin{corollary}\label{cor:radlinsat}
Let $A\in S\rx^{\al\times\al}_1$ be a linear pencil. Then
$$f\ge0\text{\ on\ }S_A\ \iff\ f\in C_A+\sqrt{\supp M_A}$$
for all $f\in\rx_1$.
\end{corollary}

\begin{proof}
Combine either Theorem \ref{indsos} with Proposition \ref{indsos++}, or Theorem \ref{sosram} with Proposition \ref{sosram++}.
\end{proof}

\begin{corollary}\label{cor:almostsat}
Let $A\in S\rx^{\al\times\al}_1$ be a linear pencil. Then
$$f\ge0\text{\ on\ }S_A\ \iff\ \forall\ep>0:f+\ep\in M_A$$
for all $f\in\rx_1$.
\end{corollary}

\begin{proof}
To prove the non-trivial implication, let $f\in\rx_1$ with $f\ge0$ on $S_A$ be given. It suffices to show $f+\ep\in M_A$ for the special case $\ep=1$
(otherwise replace $f$ by $\ep f$ and divide by $\ep$). By Corollary \ref{cor:radlinsat}, there exists $g\in C_A$, $p\in\rx$ and $k\in\N$
such that $f=g+p$ and $p^k\in\cI:=\supp M_A$. Now $f+\ep=f+1=g+(f-g)+1=g+(p+1)$ and it is enough to show that $p+1\in M_A$. This will follow
from the fact that the image of $p+1$ is a square in the quotient ring $\rx/\cI$. Indeed, since the image of $p$ in $\rx/\cI$ is nilpotent (in fact the image of
$p^k$ is zero), we can simply write down a square root of this element using the \emph{finite} Taylor expansion at $1$ of the square root function
in $1$ given by the binomial series:
\[p+1\equiv\left(\sum_{i=0}^{k-1}\binom{\frac12}ip^i\right)^2\qquad\text{mod }\cI.\qedhere\]
\end{proof}

We point out that adding an $\ep>0$ in Corollary  \ref{cor:almostsat} cannot be avoided:

\begin{example}\label{new2}
Consider $$A= \begin{bmatrix} 1 & {\tt x} \\ {\tt x} & 0\end{bmatrix}.$$
Then $S_A=\{0\}$. Hence obviously ${\tt x}\geq0$ on $S_A$. But it is easy to see that ${\tt x}\not\in M_A$ \cite[Example 2]{za}.
\end{example}

\subsection{Infeasibility of linear pencils}\label{subs:degbound}

Let $A$ be a linear pencil. In Subsection \ref{subs:weakly} we have introduced the convex cone $C_A\subseteq\rx_1$ and above we have defined the quadratic module
$M_A\subseteq\rx$ associated to $A$, both consisting of polynomials which are obviously nonnegative on the spectrahedron $S_A$.
The convex cone $C_A$ is in general too
small to detect infeasibility of $A$ (in particular, strong duality for the standard
primal-dual pair of an SDP fails; see also Subsection \ref{subs:ex} below).
On the other hand, as an easy consequence of Corollary \ref{cor:almostsat}, we have
the following version of Farkas' lemma for SDP which unlike Lemma \ref{sturm} does not only work for
strongly but also for weakly infeasible linear pencils. The price we pay is that we 
have to replace the finite-dimensional convex cone $C_A$ by the
infinite-dimensional quadratic module $M_A$. 

\begin{corollary}[Nonlinear Farkas' lemma for semidefinite programming]\label{cor:infeasibleMa}
Let $A$ be a linear pencil. Then $$S_A=\emptyset\quad\iff\quad -1\in M_A.$$
In other words, $A$ is infeasible if and only if $M_A$ is
improper.
\end{corollary}

To pass from the infinite-dimensional $M_A$ to finite-dimensions, we introduce
 truncations of $M_A$:

\begin{definition}
Given a linear pencil $A$ of size $\al$ and $k\in\N_0$, let
\begin{align*}
M_A^{(k)}&:=\Big\{\sum_ip_i^2+\sum_jv_j^*Av_j\mid p_i\in\rx_k,~v_j\in\rx^\al_k\Big\}\\
&~=\Big\{s+\tr(AS)\mid s\in\R\cx_{2k}\text{ sos-polynomial},\ S\in\R\cx^{\al\times\al}_{2k}\text{ sos-matrix}\Big\}\\
&~\subseteq\R\ac_{2k+1},
\end{align*}
be the \emph{truncated quadratic module with degree restriction} $k$
associated to $A$. Note that $M_A^{(0)}=C_A$.
\end{definition}

Our main result in this subsection is the following quantitative strengthening of Corollary \ref{cor:infeasibleMa}:

\begin{theorem}\label{thm:bounds}
Let $A$ be an infeasible linear pencil of size $\al$ in $n$ variables. Then
$$-1\in M_A^{(2^{\min\{\al{ -1},n\}}-1)}.$$
\end{theorem}

Observe that if $A$ is strictly infeasible, then this statement follows by Lemma \ref{sturm}.
We present a self-contained proof of Theorem \ref{thm:bounds} based on the following
two lemmas.

\begin{lemma}\label{lem:bounds1}
Let $A$ be an infeasible linear pencil in $n$ variables. Then
$$-1\in M_A^{(2^n-1)}.$$
\end{lemma}

\begin{proof}
We shall prove this by induction on $n$. The statement is
clear for $n=0$. Given $n\in\N$ and $$A=A_0+\sum_{i=1}^n {\tt x}_iA_i\in S\R\cx^{\al\times \al}_1,$$ we assume the statement has been
established for all infeasible linear pencils with  $n-1$ variables.

By Lemma \ref{kaffeebohne07}, there is an $u\in\R^{\al^2}\setminus\{0\}$ with $u^*(I_\al\otimes A)u=0$.
Replacing $A$ by $I_\al\otimes A$  changes neither $M_A^{(k)}$ nor $S_A=\emptyset$. Without loss of generality,
we assume therefore that there is $u\in\R^\al\setminus\{0\}$ with $u^*Au=0$. Writing $A=(\ell_{ij})_{1\le i,j\le \al}$
and performing a linear coordinate change on $\R^\al$, we can moreover assume $\ell_{11}=0$.
Furthermore, without loss of generality, $\ell_{12}\neq0$. 
Setting
$\ell':=\frac12(-1-\ell_{22})$,
$v:=[\ell'\ \ell_{12}\ 0\dots0]^*$ and $S:=vv^*$, we have
\beq\label{eq:ugly0}
\tr(AS)=v^*Av=2\ell'\ell_{12}^2+\ell_{12}^2\ell_{22}=\ell_{12}^2(\ell_{22}+2\ell')=-\ell_{12}^2\in M_A^{(1)}.
\eeq
If $\ell_{12}\in\R$, we are done. Otherwise after possibly performing an
affine linear change of variables on $\R^n$, we may assume $\ell_{12}={\tt x}_n$.

Now $A':=A({\tt x}_1,\ldots,{\tt x}_{n-1},0)$ is an 
infeasible linear pencil in $n-1$ variables.
By our induction hypothesis, $-1\in M_{A'}^{(2^{n- 1}-1)}$.
In particular, there are $$p_i\in\rx_{2^{n-{ 1}}-1}\text{ and }v_j\in\rx^{\al}_{2^{n-{ 1}}-1}$$ satisfying
$$
-1=\sum_ip_i^2+\sum_jv_j^*A'v_j.
$$
Let $q:=\sum_j v_j^*A_nv_j\in\rx_{2^{n}-2}$.
Then
\beq\label{eq:ugly1}\begin{split}
-1&=2\sum_ip_i^2+2\sum_jv_j^*A'v_j+1\\
&=2\sum_ip_i^2+2\sum_jv_j^*Av_j-2q{\tt x}_n+1\\
& = 2\sum_ip_i^2+2 \sum_jv_j^*Av_j + (1-q{\tt x}_n)^2+q^2(-{\tt x}_n^2).
\end{split}\eeq
Since $\deg q\leq 2^{n}-2$, we have $q^2(-{\tt x}_n)\in M_A^{(2^{n}-1)}$
by \eqref{eq:ugly0}.
Taken together with $(1-q{\tt x}_n)^2\in  M_A^{(2^{n}-1)}$,
\eqref{eq:ugly1} implies $-1\in M_A^{(2^{n}-1)}$.
\end{proof}

\begin{lemma}\label{lem:bounds2}
Let $A$ be an infeasible linear pencil of size $\al$. Then
\[
-1\in M_A^{(2^{\al-1}-1)}.
\]
\end{lemma}

\begin{proof}
We prove this by induction on $\al$. The statement is
clear for $\al=1$. 
Given $$A=A_0+\sum_{i=1}^n {\tt x}_iA_i\in S\R\cx^{\al\times \al}_1$$ 
of size $\al\ge2$,
we assume the statement has been
established for all infeasible linear pencils
 of size $\al-1$.
If $A$ is strongly infeasible, then $-1\in C_A=M_A^{(0)}$
by Lemma \ref{sturm}. So we may assume $A$ is weakly infeasible.

\smallskip
{\bf Claim.} 
There is an affine linear change of variables after which
$A$ assumes the form
\bes
A= \begin{bmatrix} b_0 & b^* \\ b & A'
\end{bmatrix},
\ees
where
$b_0\in\R\cx_1$,
$b=\begin{bmatrix} b_1 & \cdots & b_{\al-1}\end{bmatrix}^*\in\R\cx_1^{\al-1}$, 
$A'$ is a linear pencil of size $\al-1$, and $b_j\in\R\cx_1$ satisfy
\beq\label{eq:bj}
-b_j^2\in M_A^{(1)} \quad \text{for }j=0,\ldots,\al-1.
\eeq
Furthermore, $b_0$ can be chosen to be either $0$ or ${\tt x}_1$.

\emph{Explanation.}
By Lemma \ref{kaffeebohne07},
there is $k\in\N$ and $u_1,\ldots,u_k\in \R^\al\setminus\{0\}$ with
$\sum_{i=1}^k u_i^*Au_i=0$. We distinguish two cases.

\smallskip
{\bf Case 1.} There is $u\in\R^\al\setminus\{0\}$ with $u^*Au=0$.

\noindent
Write $A=(\ell_{ij})_{1\le i,j\le \al}$. By changing coordinates on $\R^\al$, we can assume that $u$ is the
first standard basis vector, i.e., $\ell_{11}=0$. 
Hence
$$
A=\begin{bmatrix}
0& b^* \\
b & A'
\end{bmatrix},
$$
where $b=\begin{bmatrix} b_1 & \cdots & b_{\al-1}\end{bmatrix}^*\in\R\cx_1^{\al-1}$ and 
$A'$ is a linear pencil of size $\al-1$.
As in the proof of Lemma \ref{lem:bounds1},
we deduce that
$-b_j^2\in M_A^{(1)}$ for all $j=1,\ldots,\al-1$.

\smallskip
{\bf Case 2.}
Case 1 does not apply but there are $k\ge 2$ and $u_1,\dots,u_k\in\R^\al\setminus\{0\}$ such that $\sum_{i=1}^ku_i^*Au_i=0$.

\noindent
In this case, $$\ell_{11}:=u_1^*Au_1=-\sum_{i=2}^ku_i^*Au_i\in C_A\cap-C_A=
M_A^{(0)}\cap-M_A^{(0)}.$$
Since Case 1 does not apply, $\ell_{11}\neq 0$. Furthermore, since $A$
is assumed to be weakly infeasible, $\ell_{11}\not\in\R$.
Hence
after an affine linear change of variables on $\R^n$, we
can assume $\ell_{11}={\tt x}_1$.
Thus
\bes
A=\begin{bmatrix}
{\tt x}_1& b^* \\
b & A'
\end{bmatrix},
\ees
where $b=\begin{bmatrix} b_1 & \cdots & b_{\al-1}\end{bmatrix}^*\in\R\cx_1^{\al-1}$ and 
$A'$ is a linear pencil of size $\al-1$.
Note that 
\[
-4 {\tt x}_1^2 =  (1-{\tt x}_1)^2 {\tt x}_1 +  (1+{\tt x}_1)^2 (-{\tt x}_1)
\]
shows that $-{\tt x}_1^2\in M_L^{(1)}$.
Using this, one gets similarly as above that also each of the entries $b_j$ of
$b$ satisfies $-b_j^2\in M_A^{(1)}$.
This proves our claim. \eop

\smallskip
If one of the $b_j\in\R\setminus\{0\}$, we are done
by \eqref{eq:bj}.
Otherwise we consider two cases.

\smallskip
{\bf Case a.}
The linear system $b_0(x)=0,\ b(x)=0$ is infeasible.

\noindent 
Then we proceed as follows.
There are $\ga_0,\ldots,\ga_{\al-1}\in\R$ satisfying
\beq\label{eq:fart}
\sum_{j=0}^{\al-1} \ga_j b_j = 1.
\eeq
For each $j=0,\ldots,\al-1$ and $\de\in\R$ we have
\bes
1+\de b_j = \Big( 1+\frac {\de}2 b_j\Big)^2 +
\frac{\de^2}4 (-b_j^2) \in M_A^{(1)}
\ees
by \eqref{eq:bj}.
Hence \eqref{eq:fart} implies
\[
-1 = 1-2 = 1- 2\sum_{j=0}^{\al-1} \ga_j b_j = 
\sum_{j=0}^{\al-1} \Big( \frac 1{\al} - 2\ga_jb_j\Big) \in M_L^{(1)}.
\]

\smallskip
{\bf Case b.}
The linear system $b_0(x)=b(x)=0$ is feasible.

\noindent 
Then we perform an affine linear change of variables on $\R^n$ to ensure
\[
\{x\in\R^n\mid b_0(x)=0,\ b(x)=0\} = \{0\}^r \times \R^{n-r}
\]
for some $r\in\N$.
Moreover, we may assume ${\tt x}_{1},\ldots,{\tt x}_{r}$ are among the
entries $b_j$, $j=0,\ldots,\al-1$.

Now $A'':=A'(0,\ldots,0,{\tt x}_{r+1},\ldots, {\tt x}_n)$ is an infeasible
linear pencil of size $\al-1$. By our induction hypothesis,
$-1\in M_{A''}^{(2^{\al-2}-1)}$.
In particular, there are $s\in\Si^2\cap\R\cx$ with $\deg s\leq 2^{\al-1}-2$, and
$v_i\in\rx^{\al-1}_{2^{\al-2}-1}$ satisfying
$$
-1=s+\sum_i v_i^*A'' v_i.
$$
Introducing
\[
q_t:=\sum_i v_i^*A_tv_i \in\rx_{2^{\al-1}-2}\qquad\text{and}\qquad w_i:=\begin{bmatrix}0\\v_i\end{bmatrix}\in\rx^\al_{2^{\al-2}-1}
\]
we have
\beq\label{eq:ugly3c}
\begin{split}
-1 &=\big(2s+2\sum_iv_i^*A''v_i\big)+1\\
&=\big( 2s + 2 \sum_iv_i^*A'v_i-
\sum_{t=1}^r 2q_t{\tt x}_t \big)+
\sum_{t=1}^r 
\Big( \big( \frac1{\sqrt r}-\sqrt rq_t{\tt x}_t\big)^2+2q_t{\tt x}_t+r q_t^2(-{\tt x}_t^2)\Big) \\
& = 2s + 2 \sum_{i}  w_i^*A w_i + \sum_{t=1}^r\big(\frac1{\sqrt r}-\sqrt rq_t{\tt x}_t\big)^2+\sum_{t=1}^rrq_t^2(-{\tt x}_t^2).
\end{split}\eeq
Combining 
$q_t^2(-{\tt x}_t^2)\in M_A^{(2^{\al-1}-1)}$
with $(\frac1{\sqrt r}-\sqrt rq_t{\tt x}_t)^2\in  M_A^{(2^{\al-1}-1)}$,
\eqref{eq:ugly3c} implies $-1\in M_A^{(2^{\al-1}-1)}$.
\end{proof}

\begin{proof}[Proof of Theorem {\rm\ref{thm:bounds}}]
Immediate from Lemma \ref{lem:bounds1} and Lemma \ref{lem:bounds2}.
\end{proof}

\begin{remark}
With the aid of truncated quadratic modules  associated to a linear pencil
$A$, we can introduce a hierarchy of infeasibility: $A$ is called \emph{$k$-infeasible} for $k\in\N_0$, if
$-1\in M_A^{(k)}$ and $-1\not\in M_A^{(k-1)}$. By Lemma \ref{sturm}, $A$ is strongly infeasible if and only if it is 
$0$-infeasible, and $A$ is weakly infeasible if and only if it is $k$-infeasible for some $k\in\N$.
Detecting $k$-infeasibility can be implemented as an SDP, cf.~Subsection \ref{subs:gram}.
\end{remark}

In \cite{hl06} Henrion and Lasserre extend
Lasserre's hierarchy \cite{las01} for optimizing over
scalar polynomial inequalities to polynomial matrix inequalities (PMIs).
Motivated by problems of systems control theory,
the authors of \cite{hl06} develop the primal-moment/dual-sos approach for (non-convex) PMIs,
a particular case of which are the (convex) LMIs treated here.
Our Theorem \ref{thm:bounds} shows that 
for infeasible LMIs, the SDP hierarchy described by Henrion and Lasserre
in \cite{hl06} generates a certificate of infeasibility at a \emph{finite}
relaxation order.

\subsection{Bounded spectrahedra}\label{subs:bs}

In this section we establish algebraic certificates
for  boundedness of a spectrahedron. As a corollary, we obtain a Putinar-Schm\"udgen-type
Positivstellensatz for polynomials positive on bounded spectrahedra, see Corollary
\ref{linpos}.

A quadratic module $M\subseteq\rx$ 
is said to be \textit{archimedean} if one of the following equivalent conditions holds:
\begin{enumerate}[(i)]
\item $\forall f\in\rx \ \exists N\in\N:\; N+f \in M$;
\item $\exists N\in\N:N\pm\sum_{i=1}^n{\tt x}_i^2\in M$;
\item $\forall i\in\{1,\dots,n\}\ \exists N\in\N:N\pm {\tt x}_i\in M$;
\end{enumerate}
see \cite[Corollary 5.2.4]{ma08}.

Obviously, if $M_A$ is archimedean for a linear pencil $A$, then
$S_A$ is bounded. In \cite{hkm} complete positivity (see e.g.~\cite{er00,bl04,pau02,pis03}) was used
to deduce that for \emph{strictly feasible} linear pencils the converse
holds. Subsequently, a certain generalization of this result for \emph{projections} of
spectrahedra has been proved by other techniques in \cite{gn}.
In this section we will establish the result for arbitrary bounded $S_A$. 
We deal separately with the relatively
easy case of non-empty $S_A$ (possibly with empty interior).

\begin{lemma}\label{lem:compArch}
Let $A$ be a linear pencil with $S_A\neq\emptyset$. Then
$$
S_A\text{ is bounded}\quad\iff\quad M_A\text{ is archimedean.}
$$
\end{lemma}

\begin{proof}
The direction $(\Leftarrow)$ is obvious as remarked above.
Let us consider the converse.
We first establish the existence of finitely many linear polynomials
in $C_A$ certifying the boundedness of $S_A$.

There is a ball $B\subseteq\R^n$ with $S_A\subseteq B$ and $S_A\cap\partial B=\emptyset$. For every $x\in\partial B$ there is a vector $u\in\R^n$ with 
\beq\label{eq:vecBord}
u^*A(x)u<0.
\eeq
 By continuity, \eqref{eq:vecBord} holds for all $x$ in a neighborhood
$U_x$ of $x$. From $\{U_x\mid x\in\partial B\}$ we extract by compactness a finite
subcovering 
$\{U_{x_1},\ldots, U_{x_r}\}$ of $\partial B$. Let $\ell_i:= u_i^*Au_i\in\rx_1$ and
$$
S:=\{x\in\R^n\mid \ell_1(x)\geq0,\ldots,\ell_r(x)\geq0\}.
$$
Clearly, $S_A\subseteq S$ and $S\cap \partial B=\emptyset$. Since $S_A$ is non-empty by hypothesis
and contained in $B$, it follows that $S$ contains a point of
$B$. But then it follows from the convexity of $S$ together with $S\cap \partial B=\emptyset$ that $S\subseteq B$. In particular, $S$ is bounded.

Now every $\ell_i\in C_A\subseteq M_A$. Hence the quadratic module $M$
generated by the $\ell_i$ is contained in $M_A$. 
Choose $N\in\N$ with $N\pm {\tt x}_i>0$ on $S$ for all $i$. 
Fix a $k$ and $\de\in\{-1,1\}$.
The system of linear
inequalities
$$
-N + \de x_k\geq 0, \ell_1(x)\geq0,\ldots,\ell_r(x)\geq0
$$
is infeasible. Hence by Farkas' lemma \cite{fa02}, 
there are $\al_j\in\R_{\geq0}$
satisfying
\beq\label{eq:far}
-1=\al_0(-N+\de {\tt x}_k)+\al_1\ell_1+\cdots+\al_r\ell_r.
\eeq
Note $\al_0\neq0$ since $S\neq\emptyset$. Rearranging terms
in \eqref{eq:far} yields $N-\de {\tt x}_k\in C_A$. Since $k$ and $\de$ were arbitrary and
$C_A\subseteq M_A$, we conclude that $M_A$ is archimedean.
\end{proof}

\begin{corollary}\label{cor:compArch}
Let $A$ be a linear pencil. 
Then
$$
S_A\text{ is bounded}\quad\iff\quad M_A\text{ is archimedean.}
$$
\end{corollary}
              
\begin{proof}
If $S_A\neq\emptyset$, then this is Lemma \ref{lem:compArch}.
If $S_A=\emptyset$, then $-1\in M_A$ by Corollary \ref{cor:infeasibleMa},
so $M_A$ is archimedean.
\end{proof}

\begin{rem}
Note that the above corollary is a strong variant
of Schm\"udgen's characterization \cite{sm91} of bounded basic closed
semialgebraic sets as being exactly those whose describing
finitely generated preorderings are archimedean. Preorderings have
the tendency of being much larger than quadratic modules.
In general, a finitely generated quadratic module might describe a
bounded or even an empty set without being archimedean, see
\cite[Example 6.3.1]{pd01} and \cite[Example 7.3.2]{ma08}.
Corollary~\ref{cor:compArch} says that quadratic modules associated to linear pencils behave very well
in this respect.
\end{rem}

We conclude this section with a version of Putinar's Positivstellensatz \cite{pu93} for bounded spectrahedra:

\begin{cor}\label{linpos}
Let $A$ be a linear pencil and assume that $S_A$ is bounded.
If $f\in\R\cx$ satisfies $f|_{S_A}>0$, then $f\in M_A$.
\end{cor}

\begin{proof}
By Corollary \ref{cor:compArch}, $M_A$ is archimedean. Now apply (a slight generalization of) Putinar's Positivstellensatz
\cite[Theorem 5.6.1]{ma08}.
\end{proof}

\begin{remark} Let $A$ be a linear pencil with bounded $S_A$.
\begin{enumerate}[(1)]
\item In the case $A$ is strongly feasible, Corollary \ref{linpos} has already been proved in 
\cite[\S 7]{hkm} by completely different techniques,
namely complete positivity from operator algebras. Note however that the more involved case in our approach occurs when $A$ is infeasible.
\item 
From Corollary \ref{cor:compArch} it is easy to see that the quadratic module \emph{in the sense of rings with involution} (see \cite{ks10}) associated to a
linear pencil
$A$ of size $\al$ in the ring of $\be\times\be$ matrix polynomials is archimedean (in the sense of \cite[Subsection 3.1]{ks10} or \cite[Sections 6, 7]{hkm}) if the
spectrahedron $S_A$ defined by $A$ is bounded (cf.~\cite[Section 7]{hkm}). 
Among other consequences, this implies a suitable generalization of
Corollary \ref{linpos} for \emph{matrix polynomials} positive definite on the bounded spectrahedron $S_A$ (cf.~\cite[Corollary 1]{sh06},
\cite[Theorem 3.7]{ks10} and \cite[Theorem 7.1]{hkm}).
\end{enumerate}
\end{remark}

\subsection{An application to positive linear functionals}
In this  brief subsection  we explain how our results pertain to positive
linear functionals.

\def\cR{\mathcal R}
\begin{definition}
Suppose $\cR\subseteq S\R^{\al\times\al}$ is a vector subspace, and let
$\cR_{\succeq0}:=\cR\cap S\R^{\al\times\al}_{\succeq0}$. A linear
functional $f:\cR\to\R$ is called \emph{positive} if $f(\cR_{\succeq0})\subseteq\R_{\geq0}$.
\end{definition}

Detecting positivity of a linear functional can be implemented with an SDP using
our Theorem \ref{sosram}. Let $\{A_1,\ldots,A_n\}$ be a basis for $\cR$, and introduce
the linear pencil
\[
A= A_1 {\tt x}_1+ \cdots +A_n {\tt x}_n \in S \R\cx^{\al\times\al}_1.
\]
To $f$ we associate the linear polynomial
\[
f_A = f(A_1) {\tt x}_1+ \cdots + f(A_n) {\tt x}_n \in \R\cx_1.
\]
Then $f$ is positive if and only if $f_A|_{S_A}\geq0$, and this is a situation completely
characterized by our Theorem \ref{indsos}. In turn, it can be implemented using
\hyperref[sdp:dsos]{${\rm (D^{\rm sos})} $}.

\begin{remark}
In subsequent work \cite{hkms} we shall  exploit how the ideas presented here
apply to operator algebras and complete positivity \cite{er00,bl04,pau02,pis03}, by using the results from 
\cite{hkm,hkm'}.
\end{remark}

\subsection{Examples}\label{subs:ex}
We conclude this paper with 
a series of examples pertaining to the theory developed.

\begin{example}\label{ex:ex}
The standard textbook example \cite{st00,wsv00} of a weakly infeasible linear pencil
seems to be
$$
A:= \begin{bmatrix} {\tt x} & 1 \\ 1 & 0\end{bmatrix}.
$$
Then $-1\not\in C_A$, but $-1\in M_A^{(1)}$.
Indeed,
for $u:=\begin{bmatrix}1 & -1-\frac {\tt x}2\end{bmatrix}^*,$ we have
$$-2=u^*Au.$$
Hence $A$ is $1$-infeasible.
\end{example}

\begin{example}\label{ex:sturm2}
Let 
$$
A:= \begin{bmatrix}0& {\tt x}_1 & 0 \\ {\tt x}_1 & {\tt x}_2&1\\ 0&1&{\tt x}_1\end{bmatrix}.
$$
Then $A$ is weakly infeasible and $-1\not\in M_A^{(1)}$.

Assume otherwise, and let 
\beq\label{eq:contr1}
-1=s+\sum_jv_j^*Av_j,
\eeq where
$v_j\in\rx_1^3$ and $s\in\Si^2\cap\rx_2$.
We shall carefully analyze the terms $v_j^*Av_j$. Write
$$v_j=\begin{bmatrix}q_{1j}&q_{2j}&q_{3j}\end{bmatrix}^*,
\quad \text{and}\quad q_{ij}=a_{ij}+b_{ij}{\tt x}_1+c_{ij}{\tt x}_2$$
with $a_{ij},b_{ij},c_{ij}\in\R$.
Then the ${\tt x}_2^3$ coefficient of $v_j^*Av_j$ equals
$c_{2j}^2$, so $c_{2j}=0$ for all $j$.
Next, by considering the ${\tt x}_1{\tt x}_2^2$ terms, we deduce $c_{3j}=0$.
Now the only terms possibly contributing to ${\tt x}_2^2$ come from $s$, so
$s\in\R[{\tt x}_1]_2$. The coefficient of ${\tt x}_2$ in $v_j^*Av_j$ is a square,
so $a_{2j}=0$. But now $v_j^*Av_j$ does not have a constant term
anymore, leading to a contradiction with \eqref{eq:contr1}. 

From Theorem \ref{thm:bounds} it follows that $-1\in M_A^{(3)}$. 
In fact, $-1\in M_A^{(2)}$ since
\[
-2 = u^* A u
\]
for 
$u= \begin{bmatrix}
\frac{1}{2}+ \frac{{\tt x}_2}{2}+\frac{{\tt x}_2^2}{8} & -1 & 1+\frac
   {{\tt x}_2}{2}
\end{bmatrix}^*\in \R\cx_2^3$.
Thus $A$ is $2$-infeasible.
\end{example}

\begin{example}\label{ex:new}
Let
\bes
\begin{split}
A & =\begin{bmatrix} 0 & {\tt x}_1 & 0 \\ {\tt x}_1 & {\tt x}_2 & {\tt x}_3\\0 & {\tt x}_3 & {\tt x}_1\end{bmatrix}.
\end{split}
\ees 
Then $S_{A}=\{ (0,x_2,0)\in\R^3 \mid x_2\geq0\}$ and the (affine) hyperplanes containing $S_A$ are $\{x_1=0\}$ and $\{ax_1+x_3=0\}$ ($a\in\R$).
As is shown in Case 1 of the proof of Proposition \ref{exsos}, the certificate of low-dimensionality \eqref{eq:lowdim} exists for the hyperplane
$\{x_1=0\}$, i.e., there is a quadratic sos-matrix $S$ such that $-{\tt x}_1^2=\tr(AS)$. However, none of the other hyperplanes containing $S_A$ allow
for a certificate of the form \eqref{eq:lowdim}.

Indeed, assume that there is $a\in\R$ such that $\{ax_1+x_3=0\}$ has a corresponding certificate. Combining it with the one
for $\{x_1=0\}$, we get a quadratic sos-matrix $S$ such that
$$-(2a^2){\tt x}_1^2-2(a{\tt x}_1+{\tt x}_3)^2=\tr(AS)$$
which implies
$$-{\tt x}_3^2=(2a{\tt x}_1+{\tt x}_3)^2+(-(2a^2){\tt x}_1^2-2(a{\tt x}_1+{\tt x}_3)^2)\in M_A^{(1)}.$$
Specializing ${\tt x}_3$ to $1$, one gets the contradiction $-1\in M_{A'}^{(1)}$ where $A'$ is the linear pencil from
Example~\ref{ex:sturm2}.
\end{example}

The next example gives a more explicit presentation of our dual \hyperref[sdp:dsos]{${\rm (D^{\rm sos})} $} 
applied to a classical example of an SDP with nonzero duality gap.

\begin{example}\label{ex:wesuck}
Consider
\[
A=
\begin{bmatrix} \al+ {\tt x}_2 & 0 & 0 \\ 
0 & {\tt x}_1 & {\tt x}_2 \\
0 & {\tt x}_2 & 0\end{bmatrix} =: A_0 + {\tt x}_1 A_1 + {\tt x}_2 A_2,
\]
where $\al\in\R_{>0}$ and $A_j\in S\R^{3\times3}$, and let $\ell={\tt x}_2$.
If $x\in\R^2$ satisfies $A(x)\succeq0$, then $x_2=0$. Thus 
the SDP \hyperref[sdp:p]{(P)} associated to $A$ and $\ell$ 
\beq\label{eq:ramp}
\begin{split}
\min \quad &\;  x_2 \\
%%%%%%%%%
%\text{s.t.} & \begin{bmatrix} \al+ {\tt x}_2 & 0 & 0 \\ 
%0 & {\tt x}_1 & {\tt x}_2 \\
%0 & {\tt x}_2 & 0\end{bmatrix} \succeq0
%%%%%%%%%
\text{s.t. }\; & A(x) \succeq0
%\min \{ \ell(x) \mid x\in\R^2,\, A(x)\succeq0\}
\end{split}
\eeq
has optimal value $0$. Its standard dual \hyperref[sdp:d]{(D)} is
\beq\label{eq:ramstandd}
\begin{split}
\max \quad   a \\
%\text{s.t. }  -a & = \tr \left( \begin{bmatrix} \al \\ & 0 \\ & & 0\end{bmatrix} S\right) \\
% 0 & = \tr \left( \begin{bmatrix} 0 \\ & 1 \\ & & 0\end{bmatrix} S\right) \\
% 1 & = \tr \left( \begin{bmatrix} 1 \\ & & 1 \\ & 1 & \end{bmatrix} S\right) \\
\text{s.t. }  -a & = \tr (A_0 S) \\
0 & = \tr(A_1 S) \\
1 & = \tr(A_2 S) \\
 S & \succeq0.
\end{split}
\eeq
From the second equation we see $S_{2,2}=0$ and thus (since $S\succeq0$),
$S_{2,j}=S_{j,2}=0$ for all $j$. Now the last equation implies $S_{1,1}=1$, and thus
$-a=\al$ by the first equation. Hence the optimal value of \eqref{eq:ramstandd} 
is $-\al$, yielding a duality gap of
$\al>0$.

We next present our sums of squares dual  \hyperref[sdp:dsos]{${\rm (D^{\rm sos})} $} 
for \eqref{eq:ramp}.

%\beq\label{eq:ksd}
\begin{align}
\nonumber
\max\quad&a\\ 
\nonumber
\text{s.t.}\quad&S\in S\R^{3\times3}_{\succeq0},  \ a\in\R\\
\label{eq:28}
&S_1,S_2\in S\rx^{3\times3}_2\text{ quadratic sos-matrices}\\
\nonumber
&U_1,U_2\in S\R^{3\times 3}, \; W_1,W_2\in\R^{6\times3}\\
\label{eq:29}
&\vecx1^*U_1\vecx1+\tr(AS_1)=0\\
\label{eq:30}
&U_1\succeq W_1^*W_1\\
\label{eq:31}
&\vecx1^*U_2\vecx1+\vecx2^*W_1\vecx1+\tr(AS_2)=0\\
\label{eq:32}
&U_2\succeq W_2^*W_2\\
\label{eq:33}
&\ell-a+\vecx2^*W_2\vecx1-\tr(AS)=0.
\end{align}
%\eeq
To express \eqref{eq:28} using LMI constraints, note that  
\beq\label{eq:quadsos}
D+\sum_{k=1}^2 {\tt x}_k E_k + \sum_{1\leq k\leq\ell\leq2} {\tt x}_k {\tt x}_{\ell} F_{k,\ell}
\in S\rx^{3\times3}_2 
\eeq
is an sos-matrix iff there are 
$3\times 3$ matrices $\Xi_1,\Xi_2$ and $\Phi_{1,2}$ 
satisfying 
\beq\label{eq:quadsossdp}
\begin{split}
\begin{bmatrix} D & \Xi_1 & \Xi_2 \\
\Xi_1^* & F_{1,1} & \Phi_{1,2} \\
\Xi_2^* & \Phi_{1,2}^* & F_{2,2}
\end{bmatrix} & \succeq0\\
\Xi_k+\Xi_k^* & =E_k, \quad k=1,2 \\
\Phi_{1,2}+\Phi_{1,2}^* & =F_{1,2}.
\end{split}
\eeq
Let us use the notation of \eqref{eq:quadsos} for the quadratic sos-matrix $S_1$,
and consider the left hand side of \eqref{eq:29}.  Its constant coefficient is
$\tr(A_0 D) + (U_1)_{1,1}=0$. Since $A_0,D,U_1\succeq0$, this implies $(U_1)_{1,j}=(U_1)_{j,1}=0$ and
$D_{1,j}=D_{j,1}=0$.
Next, the ${\tt x}_1^3$ term of \eqref{eq:29} is $\tr(A_1F_{1,1})=0$, whence $(F_{1,1})_{2,j}=(F_{1,1})_{j,2}=0$. From \eqref{eq:quadsossdp} it follows that $\Xi_1$ is of the form 
$\left[\begin{smallmatrix} 0 & 0 & 0 \\ * & 0 & * \\ * & 0 & *\end{smallmatrix}\right]$. Hence
$(E_1)_{2,2}=0$. Finally, by considering the ${\tt x}_1^2$ term in \eqref{eq:29},
we obtain
$\tr(A_0 F_{1,1}) + \tr(A_1 E_1) + (U_1)_{2,2}=0$. Since $(E_1)_{2,2}=0$, $A_1E_1=0$. 
As $A_0,F_{1,1},U_1\succeq0$, we deduce $U_1=\left[\begin{smallmatrix} 
0 & 0 & 0 \\ 0 & 0 & 0 \\ 0 & 0 & u\end{smallmatrix}\right]$ for some $u\in\R_{\geq0}$.
In particular, from \eqref{eq:30} we see the first two columns of $W_1$ are $0$.
Hence $\vecx2^*W_1\vecx1\in {\tt x}_2 \rx$.

Using this information on $W_1$, we can analyze \eqref{eq:31} as in 
the previous paragraph, and deduce that the first two columns of $W_2$ are $0$.
Next, we turn to \eqref{eq:33}. All its terms of degree $\geq2$ come from $\vecx2^*W_2\vecx1$,
so $(W_2)_{j,k}=0$ for all $(j,k)\neq(1,3)$. This reduces \eqref{eq:33} to the system of linear equations
\[
\begin{split}
-a & = \tr(A_0S) \\
0 & = \tr(A_1 S)\\
1+(W_2)_{1,3} &= \tr(A_2 S).
\end{split}
\]
It is instructive to compare this to \eqref{eq:ramstandd} above. Again $S_{2,j}=S_{j,2}=0$ for all $j$.
Since $\tr(A_0S)\geq0$ and we are maximizing $a$, we set $S_{1,1}=0$, yielding
$(W_2)_{1,3}=-1$ and $a=0$. It is now easy to see that this $W_2$ and $a$ can be
extended to a feasible (and thus optimal) point for the above sums of squares dual.

Finally, we give Ramana's dual for \eqref{eq:ramp};
we refer to \cite[p.~142]{ra97} for details.
\beq\label{eq:ramramd}
\begin{split}
\max \quad   a \\
%\text{s.t. }  -a & = \tr \left( \begin{bmatrix} \al \\ & 0 \\ & & 0\end{bmatrix} S\right) \\
% 0 & = \tr \left( \begin{bmatrix} 0 \\ & 1 \\ & & 0\end{bmatrix} S\right) \\
% 1 & = \tr \left( \begin{bmatrix} 1 \\ & & 1 \\ & 1 & \end{bmatrix} S\right) \\
\text{s.t. }\quad  0 & = \tr (A_j U_1), \quad j=0,1,2 \\
U_1 & \succeq W_1W_1^* \\
0 &= \tr \big( A_j (U_2+W_1 ) \big), \quad j=0,1,2 \\
U_2 & \succeq W_2W_2^* \\
0 &= \tr \big( A_j (U+W_2 ) \big), \quad j=1,2 \\
-a &=  \tr \big( A_0 (U+W_2 ) \big) \\
 U & \succeq0.
\end{split}
\eeq
The reader will have no problems verifying that the optimal value of \eqref{eq:ramramd} is
$0$. 

\end{example}

Example \ref{ex:wesuck} demonstrates that Ramana's dual is generally smaller in size than the
sums of squares dual.
However, the advantage of 
our sums of squares dual is that it admits a nice real algebraic geometric interpretation,
and naturally leads itself to the Positivstellens\"atze we presented in Section
\ref{sec:pospoly}.

\subsection*{Acknowledgments}

The authors thank three anonymous referees for their detailed reading and many helpful comments.

\end{document}